\documentclass[a4paper,twoside]{article}
\usepackage{amssymb}
\usepackage{amsmath}
\usepackage{fancyhdr}
\usepackage{amsthm}
\usepackage{rotate}
\usepackage{graphicx}
\usepackage[T1]{fontenc}

\usepackage{afterpage}  %
\theoremstyle{plain}
\newtheorem{theorem}{Theorem}[section]
\newtheorem{lemma}[theorem]{Lemma}
\newtheorem{proposition}[theorem]{Proposition}

\newtheorem{corollary}[theorem]{Corollary}

\theoremstyle{definition}

\newtheorem{example}[theorem]{Example}
\theoremstyle{remark}
\newtheorem{remark}{Remark}

\begin{document}

\afterpage{\rhead[]{\thepage} \chead[\small W. A. Dudek, R. S. Gigo\'n]{\small  Completely inverse $AG^{**}$-groupoids} \lhead[\thepage]{} }

\begin{center}
\vspace*{2pt}
{\large \textbf{\medskip Completely inverse $AG^{**}$-groupoids}}\\[3mm]
{\large \textsf{\emph{Wies{\l}aw A. Dudek, \ Roman S. Gigo\'n}}}
\\[30pt]
\end{center}
\textbf{Abstract} A completely inverse $AG^{**}$-groupoid is a groupoid satisfying the identities $(xy)z=(zy)x$, $x(yz)=y(xz)$ and $xx^{-1}=x^{-1}x$, where $x^{-1}$ is a unique inverse of $x$, that is, $x=(xx^{-1})x$ and $x^{-1}=(x^{-1}x)x^{-1}$. First we study some fundamental properties of such groupoids. Then we determine certain fundamental congruences on a completely inverse $AG^{**}$-groupoid; namely: the maximum idempotent-separating congruence, the least $AG$-group congruence and the least $E$-unitary congruence. Finally, we investigate the complete lattice of congruences of a completely inverse $AG^{**}$-groupoids.~In particular, we describe congruences on completely inverse $AG^{**}$-groupoids by their kernel and trace.\\ \\
{\footnotesize\textsf{2010 Mathematics Subject Classification: 20N02, 06B10 \\
Keywords:} completely inverse $AG^{**}$-groupoid, $AG$-group, semilattice of $AG$-groups, trace of congruence, kernel of congruence, $AG$-group congruence, $E$-unitary congruence, idempotent-separating congruence, idempotent pure congruence, fundamental congruence.}\\
Communicated by M. V. Volkov\\

\section{Introduction}

By an \emph{Abel-Grassmann's groupoid} (briefly an \emph{$AG$-groupoid}) we shall mean any groupoid which satisfies the identity $(xy)z=(zy)x$.~Such a groupoid is also called a \emph{left almost semi\-group} (briefly an \emph{$LA$-semigroup}) or a \emph{left invertive groupoid} or a \emph{right modular groupoid} (cf.\,\cite{Hol, KN, MIq}).
This structure is closely related to a commutative semigroup, because if an $AG$-groupoid contains
a right identity, then it becomes a commutative monoid. Also, if an $AG$-groupoid $A$ with a left zero $z$ is finite, then (under certain conditions) $A\setminus\{z\}$ is a commutative group \cite{MK}.

The name \emph{Abel-Grassmann's groupoids} was suggested by Stojan Bogdanovi\'c on a seminar in Ni\v{s}. First time this name appeared in the paper \cite{PS1} and in the book \cite{DK}.
 
An $AG$-groupoid $A$ satisfying the identity $x(yz)=y(xz)$ is called an \emph{$AG^{**}$-groupoid}.
Such groupoids were studied by many authors. For example, in \cite{MB} it has been proved that an $AG^{**}$-groupoid containing a left cancellative $AG^{**}$-subgroupoid can be embedded in a commutative monoid whose cancellative elements form a commutative group whose identity coincides with the identity of the commutative monoid. Also, each $AG^{**}$-groupoid satisfying the identity $(xx)x=x(xx)$ can be uniquely expressed as a semilattice of certain Archimedean $AG^{**}$-groupoids \cite{MK2}. Some other decompositions of certain $AG^{**}$-groupoids are given in \cite{P, PS}. Further, certain fundamental congruences on $AG^{**}$-groupoids are described in \cite{MK1, PB}. Finally, the kernel normal system of an inversive $AG^{**}$-groupoid has been studied in \cite{BPS}.

In this paper we investigate \emph{completely inverse $AG^{**}$-groupoids}, i.e., $AG^{**}$-groupoids
in which every element $a$ has a unique inverse $a^{-1}$ such that $aa^{-1}=a^{-1}a$.
In Section $2$ we establish some necessary definitions and facts concerning $AG^{**}$-groupoids. In Section \nolinebreak $3$ we give a few interesting results about completely inverse $AG^{**}$-groupoids. Recall from \cite{DG} that any completely inverse $AG^{**}$-groupoid satisfies Lallement's lemma for regular semigroups. Using this fact, we describe the maximum idempotent-separating congruence $\mu$ (which is equal to the least semilattice congruence) on a completely inverse $AG^{**}$-groupoid $A$.~In particular, $A$ is a semilattice $E_A$ of $AG$-groups $e\mu$  ($e\in E_A$).~Also, we show that the interval $[1_A,\mu]$ is a modular lattice.~The main result of this section says that any $AG$-groupoid $A$ is a completely inverse $AG^{**}$-groupoid if and only if $A$ is a strong semilattice of $AG$-groups.~On the one hand, in the light of this fact, we are able to construct completely inverse $AG^{**}$-groupoids.~On the other hand, completely inverse $AG^{**}$-groupoids are very similar to Clifford semigroups (i.e., (strong) semilattices of groups).

At the beginning of Section $4$ we prove that any congruence $\rho$ on a completely inverse  $AG^{**}$-groupoid is uniquely determined by $(i)$ its kernel and trace; $(ii)$ \nolinebreak the set of $\rho$-classes containing idempotents.~Furthermore, we determine the least $AG$-group congruence $\sigma$ and describe all $AG$-group congruences in terms of their kernels. Also, we give some equivalent conditions for a completely inverse $AG^{**}$-groupoid $A$ to be $E$-unitary and we describe all $E$-unitary congruences on $A$.

In Section $5$ we characterize abstractly congruences on an arbitrary completely inverse $AG^{**}$-groupoid $A$ via the so-called congruences pairs for $A$.~Furthermore, we study the trace classes of the complete lattice $\mathcal{C}(A)$ of all congruences on $A$.~The main result of this section says that the map $\rho\to\text{tr}(\rho)$ ($\rho\in\mathcal{C}(A)$) is a complete lattice homomorphism of $\mathcal{C}(A)$ onto the lattice of all congruences on the semilattice $E_A$. Also, if $\theta$ denotes the congruence on $\mathcal{C}(A)$ induced by this map, then for every $\rho\in\mathcal{C}(A)$, $\rho\theta$ is a modular lattice (with commutating elements). Moreover, $\rho\theta=[\rho_\theta,\mu(\rho)]$.~If in addition, $A$ is $E$-unitary, then $\rho_\theta=\rho\cap\sigma$, and  the mapping $\rho\to\rho\cap\sigma$ ($\rho\in\mathcal{C}(A)$) is a complete lattice homomorphism of $\mathcal{C}(A)$ onto the lattice of idempotent pure congruences.~Finally, we investigate the lattice $\mathcal{FC}(A)$ of all fundamental congruences on $A$.~We prove that $\mathcal{FC}(A)=\{\mu(\rho):\rho\in\mathcal{C}(A)\}\cong\mathcal{C}(E_A)$.

In Section $6$ we show first that each completely inverse $AG^{**}$-groupoid $A$ possesses a largest idempotent pure congruence $\tau$.~Also, we study the kernel classes of $\mathcal{C}(A)$.~We prove a result analogous to a result from the previous section. In particular, we show that the interval $[\rho\cap\mu,\tau(\rho)]$ consist of all congruences on $A$ such that their kernels are equal to $\ker(\rho)$. Further, we go back to study $E$-unitary congruences.~We determine all $E$-unitary congruences on $A$; that is, we show that a congruence is $E$-unitary if and only if its kernel is equal to the kernel of some $AG$-group congruence on $A$.~Finally, we give once again necessary and sufficient conditions for a completely inverse  $AG^{**}$-groupoid to be $E$-unitary.

The terminology used in this paper coincides with semigroup terminology  (see the book \cite{Pet}).
\section{Preliminaries}

One can easily check that in an arbitrary $AG$-groupoid $A$, the \emph{medial law} is valid, that \nolinebreak is, the equality
\begin{eqnarray}\label{medial}
(ab)(cd)=(ac)(bd)
\end{eqnarray}
holds for all $a,b,c,d\in A$.

Recall from \cite{PS} that an \emph{$AG$-band} $A$ is an $AG$-groupoid satisfying the identity $x^2=x$. If in addition, $ab=ba$ for all $a,b\in A$, then we say that $A$ is an \emph{$AG$-semilattice}.

Let $A$ be an $AG$-groupoid and $B\subseteq A$.~Denote the set of all idempotents of $B$ by $E_B$, that is, $E_B=\{b\in B:b^2=b\}$.~From $(\ref{medial})$ it follows that if $E_A \neq\emptyset$, then $E_A E_A\subseteq E_A$, therefore, $E_A$ is an $AG$-band.

An $AG$-groupoid satisfying the identity $x(yz)=y(xz)$
is said to be an \emph{$AG^{**}$-groupoid}. Any $AG^{**}$-groupoid is {\em paramedial}, i.e., it satisfies the identity
\begin{eqnarray}\label{paramedial}
(wx)(yz)=(zx)(yw).
\end{eqnarray}
Notice that each $AG$-groupoid with a left identity is an $AG^{**}$-groupoid.~Further, observe that if $A$ is an $AG^{**}$-groupoid, then (\ref{paramedial}) implies that if $E_A\neq\emptyset$, then $E_A$ is an $AG$-semi\-lattice. Indeed, in this case $E_A$ is an $AG$-band and $ef=(ee)(ff)=(fe)(fe)=fe$ for all $e,f\in E_A$. Moreover for $a,b\in A$ and $e\in E_A$, using $(1)$ and $(2)$ we have
$$
e(ab)=(ee)(ab)=(ea)(eb)=(ba)e=(ea)b.
$$

We have just proved the following result (its second part was proved earlier in \cite{PB}).

\begin{proposition}\label{semilattice} Let $A$ be an $AG^{**}$-groupoid.~Then
\begin{equation}\label{e*}
e\cdot ab=ea\cdot b
\end{equation}
for all $a,b\in A$ and $e\in E_A$.

In particular, the set of all idempotents of an arbitrary $AG^{**}$-groupoid is either empty or a semilattice.
\end{proposition}

We say that an $AG^{**}$-groupoid $A$ is \emph{completely regular} if for every $a\in A$ there exists $x\in A$ such that $a=(ax)a$ and $ax=xa$.~Observe that in such a case,
$$(ax)(ax)=(aa)(xx)=x(aa\cdot x)=x(xa\cdot a)=x(ax\cdot a)=xa=ax\in E_A,$$ therefore, $E_A$ forms a semilattice.

Let $A$ be an $AG$-groupoid with a left identity $e$ and $a\in A$. An element $a^{*}$ of $A$ is said to be a \emph{left} (\emph{right}) \emph{inverse} of $a$ if $a^{*}a=e$ (resp. $aa^{*}=e$), and an element of $A$ which is both a left and right inverse of $a$ is called an \emph{inverse} of $a$. Let $a^{*}$ be a left inverse of $a$. Then $aa^{*}=(ea)a^{*}=(a^{*}a)e=e$. It follows that any left inverse $a^{*}$ of $a$ is also its right inverse, therefore, it is its inverse. In particular, if $a^{**}$ is another left inverse of $a$, then $a^{*}=(a^{*}a)a^{*}=(a^{**}a)a^{*}=(a^{*}a)a^{**}=(a^{**}a)a^{**}=a^{**}$. The conclusion is that each left inverse of $a$ is its unique inverse. Further, if $f$ is a left identity of $A$, then $fe=e=ee$, so $e=f$, i.e., $e$ is a unique left identity of $A$. Dually, any right inverse of $a$ is its unique inverse. Denote as usual the inverse of $a$ by $a^{-1}$. Finally, it is clear that $a=(a^{-1})^{-1}$, $(ab)^{-1}=a^{-1}b^{-1}$.

An $AG$-groupoid with a left identity in which every element has a left inverse is called an \emph{$AG$-group}.

\begin{proposition}\label{AG-groups} Let $A$ be an $AG$-groupoid with a left identity $e$.~Then the following conditions are equivalent$:$

$(a)$ $A$ is an $AG$-group$;$

$(b)$ every element of $A$ has a right inverse$;$

$(c)$ every element $a$ of $A$ has a unique inverse $a^{-1};$

$(d)$ the equation $xa = b$ has a unique solution for all $a,b\in A$.
\end{proposition}

\begin{proof} By above $(a)\implies (b)\implies (c)$.

$(c)\implies (d)$.                                     
Let $a,b\in A$. Then $b=eb=(aa^{-1})b=(ba^{-1})a$, i.e., $ba^{-1}$ is a solution of the equation $xa = b$. Also, if $c$ and $d$ are solutions of this equation, then
$$c=ec=(a^{-1}a)c=(ca)a^{-1}=(da)a^{-1}=d.$$

$(d)\implies (a)$. This is obvious.
\end{proof}

Notice that if $g$ is an arbitrary idempotent of an $AG$-group $A$ with a left identity $e$, then $gg=g=eg$. Hence $e=g$, therefore, $E_A=\{e\}$.

Denote by $V(a)$ the set of all \emph{inverses} of $a$, that is,
$$
V(a)=\{a^{*}\in A:a=(aa^{*})a, \ \ a^{*}=(a^{*}a)a^{*}\}.
$$
An $AG$-groupoid $A$ is called \emph{regular} (in \cite{BPS} it is called \emph{inverse}) if $V(a)\neq\emptyset$ for all $a\in A$. Note that $AG$-groups are of course regular $AG$-groupoids, but the class of all regular $AG$-groupoids is vastly more extensive than the class of all $AG$-groups.~For example, every $AG$-band $A$ is evidently regular, since $a=(aa)a$ for every $a\in A$. In \cite{BPS} it has been proved that in any regular $AG^{**}$-groupoid, $|V(a)|=1$ $(a\in A)$, therefore, we call it an \emph{inverse $AG^{**}$-groupoid}. In that case, we denote a unique inverse of $a\in A$ by $a^{-1}$. Furthermore, recall from \cite{BPS} that in any regular $AG$-groupoid $A$, $V(a)V(b)\subseteq V(ab)$ for all $a,b\in A$.~Indeed, let $a^{*}\in V(a)$ and $b^{*}\in V(b)$. Then
$$(ab)(a^{*}b^{*})\cdot ab=(ab)a\cdot (a^{*}b^{*})b=(ab)a\cdot (bb^{*})a^{*}=(ab)(bb^{*})\cdot aa^{*}=(bb^{*}\cdot b)a\cdot aa^{*},$$
so
$$(ab)(a^{*}b^{*})\cdot ab=(ba)(aa^{*})=(aa^{*}\cdot a)b=ab.$$
By symmetry, $a^{*}b^{*}=(a^{*}b^{*})(ab)\cdot (a^{*}b^{*})$, as exactly required.~Finally, there are regular $AG$-groupoids without idempotents. On the other hand, if $a^{*}\in V(a)$ and $aa^{*}=a^{*}a$ in the $AG$-groupoid $A$, then $aa^{*}\in E_A$ (cf.\,\cite{BPS}).

\section{Completely inverse $AG^{**}$-groupoids}

One can prove (cf.\,\cite{BPS}) that in an inverse $AG^{**}$-groupoid $A$, $aa^{-1}=a^{-1}a$ if and only if $aa^{-1},a^{-1}a\in E_A$. Also, in \cite{BPS} the authors studied congruences on inverse $AG^{**}$-groupoids satisfying the identity  $xx^{-1}=x^{-1}x$. We will call such groupoids \emph{completely inverse $AG^{**}$-groupoids}. Each $AG$-group is a completely inverse $AG^{**}$-groupoid.

\begin{example}\label{E1} Let $A$ be a commutative inverse semigroup. Put $a\cdot b=a^{-1}b$ for all $a,b\in A$, where $a^{-1}$ is a unique inverse of $a$ in the inverse semigroup $A$.~Then it is easy to check that $(A,\cdot)$ is an $AG^{**}$-groupoid and $E_{(A,\cdot)}=E_A$. Furthermore, $(a\cdot a)\cdot a=a$, so $a$ is its own unique inverse in $(A,\cdot)$ for every $a\in A$, so $a\cdot a\in E_{(A,\cdot)}$ for all $a\in A$ and $(A,\cdot)$ is a completely inverse $AG^{**}$-groupoid. Also, we have that $a^{-1}\cdot (a\cdot b)= a^{-1}\cdot a^{-1}b=aa^{-1}b$. Hence
$$
a^{-1}\cdot (a^{-1}\cdot (a\cdot b))=a^{-1}\cdot aa^{-1}b=aaa^{-1}b=aa^{-1}ab=ab,
$$
that is,
$$
ab=a^{-1}\cdot (a^{-1}\cdot (a\cdot b))=a\cdot (a^{-1}\cdot (a^{-1}\cdot b))
$$
for all $a,b\in A$.

Let $\rho$ be a congruence on $(A,\cdot)$. From the above equalities follows easily that $\rho$ is a congruence on the commutative inverse semigroup $A$.~Also, if $(a,a\cdot a)\in\rho$ in $(A,\cdot)$, then $(a,a^{-1}a)\in\rho$ in $A$, since $a\cdot a=a^{-1}a$. Thus $(a^2,aa^{-1}a)\in\rho$ in $A$, so $(a^2,a)\in\rho$ in $A$. Lallement's Lemma implies that there exists $e\in E_A\cap a\rho$ and so $e\in E_{(A,\cdot)}\cap a\rho$. On the other hand, trivially $a\cdot a\in E_{a\rho}$ in $(A,\cdot)$.

Conversely, one can easily see that if $\rho$ is a congruence on $A$, then $\rho$ is also a congruence on $(A,\cdot)$. Further, if $(a,a^2)\in\rho$ in $A$, then $(a,e)\in\rho$ in $A$ for some $e\in E_A$. Since $e\cdot e=e$, then $(a,a\cdot a)\in\rho$ in $(A,\cdot)$.
\end{example}

A groupoid $A$ is said to be \emph{idempotent-surjective} if for each congruence $\rho$ on $A$, every idempotent $\rho$-class contains an idempotent of $A$. 

The following theorem was proved in \cite{DG}. Now we give another proof.

\begin{theorem}\label{i-s} Completely inverse $AG^{**}$-groupoids are idempotent-surjective.
\end{theorem}

\begin{proof} Let $\rho$ be a congruence on a completely inversive $AG^{**}$-groupoid $A$, $a\in A$ and $a\,\rho\,a^2$. We know that there exists an element $x\in A$ such that $a^2=(a^2x)a^2$, $x=(xa^2)x$ and $a^2x=xa^2\in E_A$. Note that
$$
(a^2x)(aa)=a(a^2x\cdot a)=a(xa^2\cdot a)=a(aa^2\cdot x)=(aa^2)(ax)=a^2(a^2x)=a^2(xa^2),
$$
that is, $a^2=a^2(xa^2)$. Put $e=a(xa)$. Then $e~\rho~a^2(xa^2)=a^2~\rho~a$. Also,
$$e^2=(a\cdot xa)(a\cdot xa)=a\cdot (a\cdot xa)(xa)=a\cdot (ax)(xa\cdot a)=a\cdot (ax)(a^2x).$$
Furthermore, using $(2)$
$$
(ax)(a^2x)=(ax)(xa^2)=(a^2x)(xa)=(xa^2)(xa)=(xa^2\cdot x)a
$$
by (\ref{e*}), since $xa^2\in E_A$. Hence $(ax)(a^2x)=xa$. Consequently,
$$
e^2=a(xa)=e\in E_A,
$$
as required.
\end{proof}

Let $\rho$ be a congruence on a completely inverse $AG^{**}$-groupoid $A$ and $a,b\in A$. It is evident that $(a\rho)^{-1}=a^{-1}\rho$. Hence if $(a,b)\in\rho$, then $(a^{-1},b^{-1})\in\rho$. Moreover, $A/\rho$ is a completely inverse $AG^{**}$-groupoid.

Further, let $A$ be an arbitrary groupoid and $\mathcal{V}$ be a fixed class of groupoids. We say that a congruence $\rho$ on $A$ is a \emph{$\mathcal{V}$-congruence} if $A/\rho\in\mathcal{V}$. For example, if $\mathcal{V}$ is the class of all semilattices, then $\rho$ is a \emph{semilattice} congruence on $A$ if $A/\rho$ is a semilattice.
Moreover, $A$ is called a \emph{semilattice $A/\rho$ of $AG$-groups} if there is a semilattice congruence $\rho$ on $A$ such that every $\rho$-class is an $AG$-group.~In that case, $A$ is a \emph{semilattice $Y=A/\rho$ of $AG$-groups $A_\alpha$}, $\alpha\in Y$, where $A_\alpha$ are the $\rho$-classes of $A$, or briefly a \emph{semilattice $Y=A/\rho$ of $AG$-groups $A_\alpha$}.~Notice that in such a case, $A_\alpha A_\beta\subseteq A_{\alpha\beta}$, where $\alpha\beta$ is the product of $\alpha$ and $\beta$ in $Y$. Also, $A_{\alpha\beta}=A_{\beta\alpha}$ and $A_{(\alpha\beta)\gamma}=A_{\alpha(\beta\gamma)}$.

Finally, we say that a congruence $\rho$ on a groupoid $A$ is \emph{idempotent-separating} if every $\rho$-class contains at most one idempotent of $A$.

The following simple result will at times be useful.

\begin{lemma}\label{unipotent} A completely inverse $AG^{**}$-groupoid containing only one idempotent is an $AG$-group.
\end{lemma}

\begin{proof} Let $E_A=\{e\},a\in A$. Then $e=aa^{-1}=a^{-1}a$. Hence $ea=(aa^{-1})a=a$. Thus $A$ is an $AG$-group.
\end{proof}

For elementary facts about (inverse) semigroups the reader is referred to the book of Petrich \cite{Pet}.~It is well known that each completely regular inverse semigroup is a semilattice of groups. We prove now an analogous result.

\begin{theorem}\label{mu} Let $A$ be a completely inverse $AG^{**}$-groupoid.~Define on $A$ the relation $\mu$ by
$$(a,b)\in\mu\iff aa^{-1}=bb^{-1}$$
for all $a,b\in A$. Then$:$

$(a)$ \ $\mu$ is the least semilattice congruence on $A;$

$(b)$ \ every $\mu$-class is an $AG$-group$;$

$(c)$ \ $\mu$ is the maximum idempotent-separating congruence on $A;$

$(d)$ \ $A$ is a semilattice $A/\mu$ of $AG$-groups$;$

$(e)$ \  $E_A\cong A/\mu$.

\noindent
Hence $A$ is a semilattice $E_A$ of $AG$-groups $G_e$ $(e\in E_A)$, where $G_e=\{a\in A:aa^{-1}=e\}$.
\end{theorem}

\begin{proof} $(a)$. Clearly, $\mu$ is an equivalence relation on $A$.~Let $(a,b)\in\mu$ and $c\in A$. Then
$$(ca)(ca)^{-1}=(ca)(c^{-1}a^{-1})=(cc^{-1})(aa^{-1})=(cc^{-1})(bb^{-1})=(cb)(cb)^{-1}$$
and similarly $(ac)(ac)^{-1}=(bc)(bc)^{-1}$. Hence $\mu$ is a congruence on $A$. Also,
$$
(aa^{-1})(aa^{-1})^{-1}=(aa^{-1})(a^{-1}(a^{-1})^{-1})=(aa^{-1})(a^{-1}a)=(aa^{-1})(aa^{-1})=aa^{-1},
$$
so $(a, aa^{-1})\in\mu$, where $aa^{-1}\in E_A$.~Since $E_A$ is a semilattice, then $S/\mu$ \nolinebreak is a semilattice, too.~Consequently, $\mu$ is a semilattice congruence on $A$.~Moreover, since $e^{-1}=e$ for every $e\in E_A$, then $\mu$ is idempotent-separating.

Finally, suppose that there is a semilattice congruence $\rho$ on $A$ such that $\mu\nsubseteq\rho$.~Then the relation $\mu\cap\rho$ is a semilattice congruence on $A$ which is properly contained in $\mu$, so not every $(\mu\cap\rho)$-class contains an idempotent of $A$, since each $\mu$-class contains exactly one idempotent, a contradiction with Theorem \ref{i-s}.~Consequently, $\mu$ must be the least semilattice congruence on $A$.

$(b)$. We have noticed above that $\mu$ is idempotent-separating. It is evident that every $\mu$-class is itself a completely inverse $AG^{**}$-groupoid, since $a^{-1}\in a\mu$ for all $a\in A$. In view of Lemma \ref{unipotent}, every $\mu$-class is an $AG$-group.

$(c)$. Let $\rho$ be an idempotent-separating congruence on $A,(a,b)\in\rho$.~Then $a^{-1}\,\rho\,b^{-1}$. It follows that $(aa^{-1},bb^{-1})\in\rho$. Thus $aa^{-1}=bb^{-1}$, so $(a,b)\in\mu$. Consequently, $\rho\subseteq\mu$.

The rest is obvious.
\end{proof}

Let $\mathcal{C}(A)$ denote the complete lattice of all congruences on a groupoid $A$. It is well known that if a sublattice $\mathcal{L}$ of $\mathcal{C}(A)$ has the property that $\alpha\beta=\beta\alpha$ for all $\alpha,\beta\in\mathcal{L}$, then $\mathcal{L}$ is a modular lattice.

Let $A$ be a completely inverse $AG^{**}$-groupoid.~Consider the complete lattice $[1_A,\mu]$ of all idempotent-separating congruences on $A$ (see Theorem \ref{mu}$(c)$). Let $\rho_1,\rho_2\in [1_A,\mu]$ and $(a,b)\in\rho_1\rho_2$. Then there is $c\in A$ such that $a\,\rho_1\,c\,\rho_2\,b$. In particular, $(a,c),(c,b)\in\mu$. Hence
$$
a=aa^{-1}\cdot a = cc^{-1}\cdot a~\rho_2~bc^{-1}\cdot a=ac^{-1}\cdot b~\rho_1~cc^{-1}\cdot b=bb^{-1}\cdot b=b,
$$
so $(a,b)\in\rho_2\rho_1$.~Thus $\rho_1\rho_2\subseteq\rho_2\rho_1$.~By symmetry, $\rho_2\rho_1\subseteq\rho_1\rho_2$. We have just shown the following theorem.

\begin{theorem}\label{m} Let $A$ be a completely inverse $AG^{**}$-groupoid.~Then the interval $[1_A,\mu]$, consisting of all idempotent-separating congruences on $A$, is a modular lattice.
\end{theorem}

\begin{corollary}\label{C1} The lattice of congruences on an $AG$-group is modular.
\end{corollary}

A completely inverse $AG^{**}$-groupoid $A$ is a semilattice $E_A$ of $AG$-groups $G_e$ $(e\in E_A)$, where $G_e=\{a\in A:aa^{-1}=e\}$ (Theorem \ref{mu}). The relation $\leq$ defined on the semilattice $E_A$ by $e\leq f\Leftrightarrow e=ef$ is the so-called \emph{natural partial order} on $E_A$.

Let $e\geq f$ and $a_e\in G_e$. Then $fa_e\in G_fG_e\subseteq G_{fe}=G_{f}$. Hence we may define a map $\phi_{e,f}:G_e\to G_f$ by
$$a_e\phi_{e,f}=fa_e~~(a_e\in G_e).$$
Also, for all $a_e,b_e\in G_e$, $(fa_e)(fb_e)=(ff)(a_eb_e)=f(a_eb_e)$, so
\begin{eqnarray}\label{homomorphism}
(a_e\phi_{e,f})(b_e\phi_{e,f})=(a_eb_e)\phi_{e,f},
\end{eqnarray}
i.e., $\phi_{e,f}$ is a homomorphism between the $AG$-groups $G_e$ and $G_f$. In particular, $e\phi_{e,f}=f$ (this follows also from $e\geq f$).~Observe that $\phi_{e,e}$ is the identical automorphism of the $AG$-group $G_e$.

Suppose now that $e\geq f\geq g$. Then for every $a_e\in G_e$,
$$(a_e\phi_{e,f})\phi_{f,g}=g(fa_e)=(gg)(fa_e)=(gf)(ga_e)=g(ga_e)=ga_e=a_e\phi_{e,g},$$
since $ga_e\in G_gG_e\subseteq G_{ge} = G_g$, that is,
\begin{eqnarray}\label{system}
\phi_{e,f}\phi_{f,g}=\phi_{e,g}
\end{eqnarray}
for every $e,f,g\in E_A$ such that  $e\geq f\geq g$.

Finally, let $a_e\in G_e$ and $a_f\in G_f$ (and so $a_ea_f\in G_{ef}$; also $e,f\geq ef$).~Then we get
$a_ea_f=(ef)(a_ea_f)=(ef\cdot ef)(a_ea_f)=((ef)a_e)((ef)a_f)$, i.e.,
\begin{eqnarray}\label{operation}
a_ea_f=(a_e\phi_{e,ef})(a_f\phi_{f,ef}).
\end{eqnarray}

Remark that we have used only the medial law in the proof of the equalities (\ref{homomorphism}), (\ref{system}) and (\ref{operation}), therefore, if an $AG$-groupoid $A$ is a semilattice $E_A$ of the $AG$-groups $G_e$ ($e\in E_A$), then these equalities hold true.

Let now $Y$ be a semilattice, $\mathcal{F}=\{A_\alpha:\alpha\in Y\}$ be a family of disjoint $AG$-groupoids of type $\mathcal{T}$, indexed by the set $Y$ ($\mathcal{F}$ may be a family of disjoint $AG$-groups). Suppose also that for each pair $(\alpha,\beta)\in Y\times Y$ such that $\alpha\geq\beta$ there is an associated homomorphism $\phi_{\alpha,\beta}:A_\alpha\to A_\beta$ such that

$(a)$ \ $\phi_{\alpha,\alpha}$ is the identical automorphism of $A_\alpha$ for every $\alpha\in Y$, and

$(b)$ \ $\phi_{\alpha,\beta}\phi_{\beta,\gamma}=\phi_{\alpha,\gamma}$ for all $\alpha,\beta,\gamma\in Y$ such that $\alpha\geq\beta\geq\gamma$.

\medskip\noindent
Put $A=\bigcup\{A_\alpha:\alpha\in Y\}$, and define a binary operation $\cdot$ on $A$ by the rule that if $a_\alpha\in A_\alpha$ and $a_\beta\in A_\beta$, then
$$a_\alpha\cdot a_\beta=(a_\alpha\phi_{\alpha,\alpha\beta})(a_\beta\phi_{\beta,\alpha\beta}),$$
where the multiplication on the right side takes place in the $AG$-groupoid $A_{\alpha\beta}$.

It is a matter of routine to check that $(A,\cdot)$ is an $AG$-groupoid.~If in addition, each $AG$-groupoid $A_\alpha$ is an $AG^{**}$-groupoid (in particular, an $AG$-group), then $(A,\cdot)$ is itself an $AG^{**}$-groupoid.~Finally, in the light of the condition $(a)$, the new multiplication coincides with the given of each $A_\alpha$, so $A$ is certainly a semilattice $Y$ of $AG$-groupoids $A_\alpha$.~We usually denote the product in $A$ also by juxtaposition, and write $A=[Y;A_\alpha;\phi_{\alpha,\beta}]$.

We call the $AG$-groupoid $[Y;A_\alpha;\phi_{\alpha,\beta}]$ a \emph{strong semilattice of $AG$-groupoids $A_\alpha$}.

In fact, we have proved the following theorem (see (\ref{homomorphism}), (\ref{system}) and (\ref{operation})).

\begin{theorem}\label{strong semilattice} Let an $AG$-groupoid $A$ be a semilattice $A/\rho$ of $AG$-groups.~Then $A$ is a strong semilattice of $AG$-groups.~In fact,
$$A=[E_A;G_e;\phi_{e,f}],$$
where for all $e,f\in E_A,$ $G_e=e\rho;$ $\phi_{e,f}:G_e\to G_f$ is given by
$$a_e\phi_{e,f}=fa_e~~(a_e\in G_e),$$
and
$$a_ea_f=(a_e\phi_{e,ef})(a_f\phi_{f,ef})~~(a_e\in G_e,a_f\in G_f).$$
In particular, $A$ is an $AG^{**}$-groupoid.
\end{theorem}

\begin{proof}Let $A$ be a semilattice $A/\rho$ of $AG$-groups, then $\rho$ is idempotent-separating. Hence $E_A\cong A/\rho$, so $E_A$ is necessarily a semilattice.~Thus $A$ is a semilattice $E_A$ of $AG$-groups $G_e=e\rho$ ($e\in E_A$).~This implies the thesis of the theorem.
\end{proof}

It is well known that if a semigroup $S$ is a semilattice of groups, then its idempotents are \emph{central}, that is, $se=es$ for all $s\in S$ and $e\in E_S$.~The following proposition \nopagebreak says particularly that there is no non-associative $AG$-groupoids which are a semilattice of $AG$-groups and their idempotents are central.

\begin{proposition}\label{abelian} Let $A$ be an $AG$-groupoid which is a semilattice of $AG$-groups.~If \nolinebreak the idempotents of $A$ are central, then $A$ is a strong semilattice of abelian groups. In particular, $A$ is a commutative semigroup.
\end{proposition}

\begin{proof} Let $A=[E_A;G_e;\phi_{e,f}]$. If the idempotents of $A$ are central, then particularly for all $e\in E_A$, $ae=ea$ for every $a\in G_e$.~This implies that every $G_e$ is a commutative group, so $A$ is a strong semilattice of abelian groups. From the definition of the multiplication in $[E_A;G_e;\phi_{e,f}]$ and from the fact that abelian groups are commutative semigroups follows that $A$ is a commutative semigroup.
\end{proof}

\begin{remark} Let $A$ be a completely inverse $AG^{**}$-groupoid.~Then $ae=ea$ for all $a\in A,$ $e\in E_A$ if and only if $a=a(a^{-1}a)$ for every $a\in A$. Indeed,
$$ea=e(aa^{-1}\cdot a)=(e\cdot aa^{-1})a=(a\cdot aa^{-1})e=(a(a^{-1}a))e.$$
This implies that if $a=a(a^{-1}a)$, then the idempotents of $A$ are central.~The converse implication is obvious.

In the proof of Theorem \ref{i-s} we have shown that $a^2=a^2(xa^2)$ for every $a\in A$, where $x\in V(a^2)$. Furthermore, $A^{(2)}=\{a^2:a\in A\}$ is an $AG^{**}$-groupoid, since $a^2b^2=(ab)^2$ for all $a,b\in A$.~Also, $(a^{-1})^2\in V(a^2)$ for every $a\in A$.~Evidently, $E_A\subseteq A^{(2)}$.~Consequently, $A^{(2)}$ is a completely inverse $AG^{**}$-groupoid in which the idempotents are central.~From Proposition \ref{abelian} we obtain the following theorem.
\end{remark}

\begin{theorem}\label{A^2} If $A$ is a completely inverse $AG^{**}$-groupoid, then $A^{(2)}$ is a strong semi\-lattice of abelian groups with semilattice $E_A$ of idempotents.
\end{theorem}

The next theorem gives necessary and sufficient conditions for an $AG$-groupoid to be a completely inverse $AG^{**}$-groupoid.

\begin{theorem}\label{major} The following conditions concerning an $AG$-groupoid $A$ are equivalent$:$

$(a)$ \ $A$ is a completely inverse $AG^{**}$-groupoid$;$

$(b)$ \ $A$ is a semilattice of $AG$-groups$;$

$(c)$ \ $A$ is a strong semilattice of $AG$-groups.
\end{theorem}

\begin{proof}$(a)\implies (b)$ by Theorem \ref{mu} and $(b)\implies (c)$ by Theorem \ref{strong semilattice}.

$(c)\implies (a)$. In that case, $A$ is an $AG^{**}$-groupoid (see again Theorem \ref{strong semilattice}). Also, let $a\in A$. Then $a$ belongs to some $AG$-group $G_e$, where $e$ is a left identity of $G_e$. Consider now a unique inverse $a^{-1}$ of $a$ in $G_e$. Then evidently $a=(aa^{-1})a,$ $a^{-1}=(a^{-1}a)a^{-1}$ and $aa^{-1}=a^{-1}a=e$. Consequently, $A$ is a completely inverse $AG^{**}$-groupoid.
\end{proof}

\begin{remark} In view of the above theorem, we are able to construct completely inverse $AG^{**}$-groupoids.
\end{remark}

Let $A$ be a completely inverse $AG^{**}$-groupoid.~The relation
$\leq_A$ defined on $A$ by $a\leq_A b$ if $a\in E_Ab$ is the
\emph{natural partial order} on $A$. Notice the restriction of
$\leq_A$ to $E_A$ is equal to the natural partial order $\leq$ on
$E_A$, therefore, we will be write briefly $\leq$ instead of
$\leq_A$.

\medskip

The following result can be deduced from \cite{P}.

\begin{lemma}\label{<} In any completely inverse $AG^{**}$-groupoid $A$, the relation $\leq$ is a compatible partial order on $A$.~Also, $a\leq b$ implies $a^{-1}\leq b^{-1}$ for all $a,b\in A$.
\end{lemma}

\begin{proof}We include a simple proof. It is evident that $\leq$ is reflexive and preserves inverses. Let $a\leq b$ and $b\leq a$, i.e., $a=eb$ and $b=fa$ for some $e,f\in E_A$.~Then by Proposition \ref{semilattice}, $ea=a$.~Using again Proposition \ref{semilattice}, $a=eb=e(fa)=(ef)a=(fe)a=f(ea)=fa=b$. Hence $\leq$ is antisymmetric. From Proposition \ref{semilattice} it follows also that $\leq$ is transitive. Finally, if $a\leq b$ and $c\leq d$, that is, $a=eb$ and $c=fd$ for some $e,f\in E_A$, then we obtain that $ac=(eb)(fd)=(ef)(bd)$.~Thus $ac\leq bd$.
\end{proof}

For some equivalent definitions of the relation $\leq$, consult \cite{P}. Moreover, we have the following proposition.

\begin{proposition} In any completely inverse $AG^{**}$-groupoid $A$, $\leq\cap$ $\mu=1_A$, that is, if $A=[E_A;G_e;\phi_{e,f}]$, then $\leq_{|G_e}=1_{G_e}$ for every $e\in E_A$.
\end{proposition}

\begin{proof} Let $a(\leq\cap$ $\mu)b$. Then $aa^{-1}=bb^{-1}$ and $a=eb$ for some $e\in E_A$, therefore we get $aa^{-1}=(eb)(eb^{-1})=(ee)(bb^{-1})=e(bb^{-1})=(eb)b^{-1}=ab^{-1}$. Consequently,
$$a=(aa^{-1})a=(bb^{-1})a=(ab^{-1})b=(aa^{-1})b=(bb^{-1})b=b,$$
as required.
\end{proof}

Finally, for any nonempty subset $B$ of a completely inverse $AG^{**}$-groupoid $A$, we call
$$B\omega=\{a\in A:\exists~(b\in B)~b\leq a\}$$
the \emph{closure of $B$ in $A$}; if $B=B\omega$, then $B$ is \emph{closed in $A$}. Note that $B\omega$ is closed in $A$.

It is clear that a subgroupoid $B$ of a completely inverse $AG^{**}$-groupoid $A$ is itself  a completely inverse $AG^{**}$-groupoid if and only if $b\in B$ implies $b^{-1}\in B$ for every $b\in B$. In such a case, $B$ is a \emph{completely inverse $AG^{**}$-subgroupoid} of $A$.~Using Lemma \ref{<}, one can prove the following proposition.

\begin{proposition}\label{omega}If $B$ is a completely inverse $AG^{**}$-subgroupoid of a completely inverse $AG^{**}$-groupoid $A$, then $B\omega$ is a closed completely inverse $AG^{**}$-subgroupoid of $A$.
\end{proposition}

In particular, $E_A\omega$ is a closed completely inverse $AG^{**}$-subgroupoid of $A$.~It is easy to see that
$$
E_A\omega=\{a\in A:(\exists\,e\in E_A)~ea\in E_A\}.
$$

\section{Certain $E$-unitary congruences}

Let $\rho$ be a congruence on a completely inverse $AG^{**}$-groupoid $A$.~By the \emph{kernel} ker($\rho$) (respectively the \emph{trace} tr($\rho$)) of $\rho$ we shall mean the set $\{a\in A: (a,a^2)\in\rho\}$ (respectively the restriction of $\rho$ to the set $E_A$). Note that tr($\rho$) is a congruence on the semilattice \nolinebreak $E_A$. Also, in the light of Theorem \ref{i-s},
$$
\ker(\rho)=\{a\in A:\exists~(e\in E_A)\,(a,e)\in\rho\}=\bigcup\{e\rho:e\in E_A\}.
$$

The following proposition may be sometimes useful.

\begin{proposition}\label{ab=ba} Let $A=[E_A;G_e;\phi_{e,f}]$ be a completely inverse $AG^{**}$-groupoid and let $a,b\in A$ be such that $ab\in E_A$.~Then $ab=ba$.
\end{proposition}

\begin{proof}Let $ab=e\in E_A$. Then
$$ba=b(aa^{-1}\cdot a)=(aa^{-1})(ba)=(ab)(a^{-1}a)\in E_A.$$
Since $ab,ba\in G_e$, then $ab=ba$.
\end{proof}

The following theorem says particularly that each congruence on a completely inverse $AG^{**}$-groupoid is uniquely determined by its kernel and trace.

\begin{theorem}\label{kernel-trace} If $\rho$ is a congruence on a completely inverse $AG^{**}$-groupoid $A$, then
$$
(a,b)\in\rho\iff (aa^{-1},bb^{-1})\in{\rm tr}(\rho)~\&~ab^{-1}\in\ker(\rho).
$$
Thus for all $\,\rho_1,\rho_2\in\mathcal{C}(A)$,
$$
\rho_1\subseteq\rho_2\iff{\rm tr}(\rho_1)\subseteq{\rm tr}(\rho_2)~\&~\ker(\rho_1)\subseteq\ker(\rho_2).
$$
In particular, each congruence on a completely inverse $AG^{**}$-groupoid is uniquely determined by its kernel and trace.
\end{theorem}

\begin{proof} Let $(a,b)\in\rho$. Then evidently $(a^{-1},b^{-1}),(ab^{-1},bb^{-1})\in\rho$, so $(aa^{-1},bb^{-1})\in\text{tr}(\rho)$ and $ab^{-1}\in\ker(\rho)$.

Conversely, let now $(aa^{-1},bb^{-1})\in\text{tr}(\rho),ab^{-1}\in\text{ker}(\rho)$.~In view of Theorem \ref{mu}, $(a\rho,b\rho)\in\mu_{S/\rho}$, so $((ab^{-1})\rho,(bb^{-1})\rho)\in\mu_{S/\rho}$.~Since $ab^{-1}\in\text{ker}(\rho)$, then $(ab^{-1})\rho\in E_{A/\rho}$. Evidently, $(bb^{-1})\rho\in E_{A/\rho}$.~Hence $(ab^{-1})\rho=(bb^{-1})\rho$ (by Theorem \ref{mu}$(c)$).~Thus
$$a\rho=(aa^{-1}\cdot a)\rho=(bb^{-1}\cdot a)\rho=(ab^{-1}\cdot b)\rho=(bb^{-1}\cdot b)\rho=b\rho,$$
as required.~The rest of the theorem follows from the first equivalence.
\end{proof}

\begin{remark}\label{Clifford} Note that the first part of the above theorem is true for an arbitrary Clifford semigroup, the proof is very similar.~In fact, if $ab^{-1}\in\text{ker}(\rho)$, then $$(ab^{-1})\rho=(b^{-1}a)\rho=(b^{-1}b)\rho,$$
so $a\rho=(aa^{-1}\cdot a)\rho=(bb^{-1}\cdot a)\rho=(b\cdot b^{-1}a)\rho=(b\cdot b^{-1}b)\rho=b\rho$.

Clearly, the condition $ab^{-1}\in\text{ker}(\rho)$ from Theorem \ref{kernel-trace} is equivalent to the condition $a^{-1}b\in\text{ker}(\rho)$. In the light of Proposition \ref{ab=ba}, it is also equivalent to $b^{-1}a\in\text{ker}(\rho)$.
\end{remark}

\begin{theorem}\label{idempotent classes} Let $\rho_1,\rho_2$ be congruences on a completely inverse $AG^{**}$-groupoid $A$.~Then the following statements are equivalent$:$

$(a)$ \ $e\rho_1\subseteq e\rho_2$ for every $e\in E_A;$

$(b)$ \ $\rho_1\subseteq\rho_2$.

\noindent
In particular, every congruence $\rho$ on a completely inverse $AG^{**}$-groupoid is uniquely determined by the set of $\rho$-classes containing idempotents.
\end{theorem}

\begin{proof} $(a)\implies (b)$. Let $a\in b\rho_1$. Then
$$aa^{-1}\in (bb^{-1})\rho_1\subseteq (bb^{-1})\rho_2~~\&~~ab^{-1}\in (bb^{-1})\rho_1\subseteq (bb^{-1})\rho_2.$$
In the light of Theorem \ref{kernel-trace}, $a\in b\rho_2$, that is, $\rho_1\subseteq\rho_2$.

$(b)\implies (a)$. This is trivial.
\end{proof}

In Section $5$ we shall characterize abstractly the congruences on a completely inverse $AG^{**}$-groupoid $A$ via the congruence pairs for $A$.

A nonempty subset $B$ of a groupoid $A$ is called \emph{left} (\emph{right}) \emph{unitary} if $ba\in B$ (resp. $ab\in B$) implies $a \in B$ for every $b \in B,a\in A$. Also, we say that $B$ is \emph{unitary} if it is both left and right unitary. Finally, a groupoid $A$ is said to be \emph{$E$-unitary} if $E_A$ is unitary.

\begin{proposition}\label{E-unitary} Let $E_A$ be a left unitary subset of an $AG$-groupoid.~Then $E_A$ is also right unitary.~If in addition, $A$ is an $AG^{**}$-groupoid, then the following conditions are equivalent$:$

$(a)$ \ $A$ is $E$-unitary$;$

$(b)$ \ $E_A$ is left unitary$;$

$(c)$ \ $E_A$ is right unitary.
\end{proposition}

\begin{proof} $(a)\implies (b),(c)$. Obvious.

$(b)\implies (a)$. Let $a\in A,e\in E_A$ and let $ae=f\in E_A$. Then $(ae)f\in E_A$, therefore, $(fe)a\in E_A$. Thus $a\in E_A$, since $fe\in E_A$ and $E_A$ is left unitary.

$(c)\implies (a)$. Let $a\in A,e\in E_A$ and $ea=f\in E_A$. Then, using $(3)$,
$$f=f(ea)=(fe)a=(ae)f.$$
Hence $ae\in E_A$.
Thus $a\in E_A$.
\end{proof}

$AG$-groups are examples of $E$-unitary completely inverse $AG^{**}$-groupoids.

A congruence $\rho$ on a completely inverse $AG^{**}$-groupoid is  a \emph{$AG$-group} congruence if $A/\rho$ is an $AG$-group.~By Lemma \ref{unipotent}, $\rho$ is an $AG$-group congruence if and only if tr$(\rho)=E_A\times E_A$.~Since $A\times A$ is an $AG$-group congruence on $A$, then the intersection of all the $AG$-group congruences on $A$ is the least $AG$-group congruence on $A$.

A more useful characterization of the least $AG$-group congruence on $A$ is given in the following theorem.

\begin{theorem}\label{sigma} In any completely inverse $AG^{**}$-groupoid $A$,
$$
\sigma=\{(a,b)\in A\times A:(\exists~e\in E_A)\,ea=eb\}
$$
is the least $AG$-group congruence with the kernel $E_A\omega$.
\end{theorem}

\begin{proof} It is evident that $\sigma$ is reflexive and symmetric. Let $(a,b),(b,c)\in\sigma$, so that $ea=eb$ and $fb=fc$ for some $e,f\in E_A$. Using Proposition \ref{semilattice}, we have that
$$
(fe)a=f(ea)=f(eb)=(fe)b=(ef)b=e(fb)=e(fc)=(ef)c=(fe)c,
$$
where $fe\in E_A$. Thus $(a,c)\in\sigma$. Consequently, $\sigma$ is an equivalence relation on $A$.~Further, let $(a,b)\in\sigma$, that is, $ea=eb$, where $e\in E_A$, and let $c\in A$. Then again in the light of Proposition \ref{semilattice}, $e(ac)=(ea)c=(eb)c=e(bc)$. Also,
$$(cc^{-1})e\cdot ca=(cc^{-1})c \cdot ea =(cc^{-1})c \cdot eb=(cc^{-1})e \cdot cb,$$
where $(cc^{-1})e\in E_A$. Hence $\sigma$ is a congruence on $A$. Since $(ef)e=(ef)f$ and $ef\in E_A$ for all $e,f\in E_A$, then $\sigma$ is an $AG$-group congruence on $A$. Also, let $\rho$ be an $AG$-group congruence on $A$ and $(a,b)\in\sigma$. Then $ea=eb$, where $e\in E_A$, so $(e\rho)(a\rho)=(e\rho)(b\rho)$. Hence $a\rho=b\rho$, since $e\rho$ is a left identity of the $AG$-group $A/\rho$. Thus $\sigma\subseteq\rho$. Consequently, $\sigma$ is the least $AG$-group congruence on $A$. Finally,
$$
a\in\text{ker}(\sigma)\Longleftrightarrow (\exists\,f\in E_A)~(a,f)\in\sigma\Longleftrightarrow (\exists\,e,f\in E_A)~ea=ef\Longleftrightarrow a\in E_A\omega,
$$
as required.
\end{proof}

From Theorem \ref{kernel-trace} follows that $(a,b)\in\sigma\Leftrightarrow ab^{-1}\in E_A\omega$. Also, in the light of the end of Section $3$, $E_A\omega$ is a closed completely inverse $AG^{**}$-subgroupoid of $A$. Evidently, $E_A\subseteq E_A\omega$ and if $ab\in E_A\omega$, then $ba\in E_A\omega$.

\medskip
A nonempty subset $B$ of a completely inverse $AG^{**}$-groupoid $A$ is called:

$(F)$ \ \emph{full} if $E_A\subseteq B$;

$(S)$ \ \emph{symmetric} if $xy\in B$ implies $yx\in B$ for all $x,y\in A$.

\medskip\noindent
 A completely inverse $AG^{**}$-subgroupoid $N$ of $A$ is said to be \emph{normal} if it full, closed and symmetric. In that case, we shall write $N\lhd A$.

Denote the set of all $AG$-group congruences on a completely inverse $AG^{**}$-group\-oid $A$ by $\mathcal{GC}(A)$. It is clear that $\mathcal{GC}(A)=[\sigma,A\times A]$ is a complete sublattice of $\mathcal{C}(A)$. Note that $\mathcal{GC}(A)\cong\mathcal{C}(A/\sigma)$ and so the lattice $\mathcal{GC}(A)$ is modular (by Corollary \ref{C1}).  Further, let $\mathcal{N}(A)$ be the set of all normal completely inverse $AG^{**}$-subgroupoids of $A$.~It is obvious that $E_A\omega\subseteq N$ for every $N\lhd A$, and if $\emptyset\neq\mathcal{F}\subseteq\mathcal{N}(A)$, then $\bigcap \{B\,:\,B\in\mathcal{F}\}\in\mathcal{N}(A)$. Consequently, $\mathcal{N}(A)$ is a complete lattice.

The following theorem (proved in \cite{P1}) describes the $AG$-group congruences on a completely inverse $AG^{**}$-groupoid in the terms of its normal completely inverse $AG^{**}$-subgroupoids.

\begin{theorem}\label{AG-group congruences} Let $A$ be a completely inverse $AG^{**}$-groupoid, $N\lhd A$.~Then the relation
$$
\rho_N=\{(a,b)\in A\times A:ab^{-1}\in N\}
$$
is the unique $AG$-group congruence $\rho$ on $A$ for which $\ker(\rho)=N$.

Conversely, if $\rho\in\mathcal{GC}(A)$, then $\ker(\rho)\in\mathcal{N}(A)$ and $\rho=\rho_N$ for $N=\ker(\rho)$.

Consequently, the map $\phi:\mathcal{N}(A)\to\mathcal{GC}(A)$ given by $N\phi=\rho_N$ $(N\in\mathcal{N}(A))$ is a complete lattice isomorphism of $\mathcal{N}(A)$ onto $\mathcal{GC}(A)$.~In particular, the lattice  $\mathcal{N}(A)$ is modular.\qed
\end{theorem}

We say that a congruence $\rho$ on a groupoid $A$ is \emph{idempotent pure} if $e\rho\subseteq E_A$ for all $e\in E_A$. Notice that any idempotent pure congruence $\rho$ on an arbitrary completely inverse $AG^{**}$-groupoid $A$ is contained in $\sigma$. Indeed, if $(a,b)\in\rho$, then $(ab^{-1},bb^{-1})\in\rho$, so $ab^{-1}\in E_A\subseteq E_A\omega$. Thus $(a,b)\in\sigma$, as required.

The following theorem gives necessary and sufficient conditions for a completely inverse $AG^{**}$-groupoid to be $E$-unitary.

\begin{theorem}\label{E-unitary2} Let $A=[E_A;G_e;\phi_{e,f}]$ be a completely inverse $AG^{**}$-groupoid. Then the following conditions are equivalent$:$

$(a)$ \ $A$ is $E$-unitary$;$

$(b)$ \ $\ker(\sigma)=E_A;$

$(c)$ \ $\sigma$ is the maximum idempotent pure congruence on $A;$

$(d)$ \ $\sigma\cap\mu=1_A;$

$(e)$ \ $\phi_{e,f}$ is a monomorphism for all $e,f\in E_A$ such that $e\geq f$.
\end{theorem}

\begin{proof} In view of Proposition \ref{E-unitary}, $(a)$ and $(b)$ are equivalent, since $\ker(\sigma)=E_A\omega$.

$(b)\implies (c)$. This follows from the preceding remark.

$(c)\implies (d)$. Indeed, tr$(\sigma\cap\mu)\subseteq\text{tr}(\mu)=1_{E_A}$ (by Theorem \ref{mu}$(c)$). Furthermore, ker$(\sigma\cap\mu)\subseteq\text{ker}(\sigma)=E_A$. In the light of Theorem \ref{kernel-trace}, $\sigma\cap\mu=1_A$.

$(d)\implies (e)$. Let $a_e,b_e\in G_e$ be such that $a_e\phi_{e,f}=b_e\phi_{e,f}$.~Then $fa_e=fb_e$, therefore, $(a_e,b_e)\in\sigma$. Since clearly $(a_e,b_e)\in\mu$, then $a_e=b_e$.

$(e)\implies (a)$. Let $a_f\in G_f$ be such that $ea_f=g$ ($e,f,g\in E_A$). Then $ef=g$. Hence $eg=g$, that is, $e\geq g$, so $a_f\phi_{e,g}=g\phi_{e,g}$, therefore, $a_f=g\in E_A$. Thus $A$ is $E$-unitary (by Proposition \ref{E-unitary}).
\end{proof}

Let $\rho, \upsilon$ be congruences on $A$ such that $\rho \subseteq \upsilon$.~Then the map $\Phi:A/\rho \rightarrow A/\upsilon,$ where $(a\rho)\Phi=a\upsilon$ for every $a\in A$, is a well-defined epimorphism between these groupoids. Denote its kernel by
$$
\upsilon /\rho = \{(a\rho, b\rho)\in A/\rho \times A/\rho: (a,b)\in\upsilon \}.
$$
Then $(A/\rho)/(\upsilon /\rho) \cong A/\upsilon$. Moreover, every congruence $\alpha$ on $A/\rho$ is of the form  $\upsilon/\rho$, where $\upsilon \supseteq \rho$ is a congruence on $A$. Indeed, the relation $\upsilon$, defined  on $A$ by: $(a, b)\in \upsilon$ if and only if $(a\rho, b\rho)\in \alpha$, is a congruence on $A$ such that  $\rho \subseteq \upsilon$ and $\alpha = \upsilon /\rho$.

We are able now to determine all $E$-unitary congruences on any completely inverse $AG^{**}$-groupoid.

\begin{theorem}\label{E-unitary congruences} The intersection of an $AG$-group congruence and a semilattice congruence on a completely inverse $AG^{**}$-groupoid $A$ is an $E$-unitary congruence on $A$.
Moreover, any $E$-unitary congruence on a completely inverse $AG^{**}$-groupoid $A$ can be expressed uniquely in this way.
\end{theorem}

\begin{proof} Let $\rho_N$ be an $AG$-group congruence ($N\lhd A$) and $\upsilon$ be a semilattice congruence on $A$. Put for simplicity $\rho=\rho_N\cap\upsilon$, and observe that $\rho_N/\rho$ is an $AG$-group congruence on $A/\rho$ and $\upsilon/\rho$ is a semilattice congruence on $A/\rho$. Since $\rho_N/\rho\cap\upsilon/\rho=1_{A/\rho}$, then $\sigma_{A/\rho}\cap\mu_{A/\rho}=1_{A/\rho}$ (see Theorem \ref{mu}$(a)$). In the light of Theorem \ref{E-unitary2}, $\rho$ is an $E$-unitary congruence on $A$.

Conversely, let $\rho$ be an $E$-unitary congruence on $A$, $\rho_N/\rho=\sigma_{A/\rho}$ and let $\upsilon/\rho=\mu_{A/\rho}$, where $\rho\subseteq\rho_N,\upsilon$. Then $\rho_N$ is an $AG$-group congruence and $\upsilon$ is a semilattice congruence on $A$. Also, $(\rho_N\cap\upsilon)/\rho=\sigma_{A/\rho}\cap\mu_{A/\rho}=1_{A/\rho}$ (again by Theorem \ref{E-unitary2}). Thus $\rho=\rho_N\cap\upsilon$, as required.

Finally, let $\rho=\rho_{N_1}\cap\upsilon_1=\rho_{N_2}\cap\upsilon_2$, where $N_i \lhd A$ and $\upsilon_i$ is a semilattice congruence on $A$ ($i = 1, 2$). Let $(a, b) \in \upsilon_1$.~Since $\upsilon_1 \cap \upsilon_2$ is a semilattice congruence on $A$, then there exists $e,f\in E_A$ such that $(a, e)\in\upsilon_1\cap\upsilon_2,$ $(e, f)\in\rho_{N_1},(f, b)\in\upsilon_1\cap\upsilon_2$ (Theorem \ref{i-s}), so $(e, f)\in\upsilon_1\cap\rho_{N_1} =\upsilon_2\cap\rho_{N_2}\subseteq\upsilon_2$. Hence $(a, b)\in\upsilon_2$, i.e., $\upsilon_1\subseteq\upsilon_2$. By symmetry, we deduce that $\upsilon_1 = \upsilon_2$. Put $\upsilon_1 = \upsilon_2 = \upsilon$, so that  $\rho=\rho_{N_1}\cap\upsilon=\rho_{N_2}\cap\upsilon$. If $(a, b)\in\rho_{N_1}$, then $(aab, abb)\in \upsilon\cap\rho_{N_1}\subseteq\rho_{N_2}$, therefore, $(a, b)\in\rho_{N_2}$ (by cancellation).~Hence $\rho_{N_1}\subseteq\rho_{N_2}$. By symmetry,  $\rho_{N_2}\subseteq\rho_{N_1}$. Thus  $\rho_{N_1}=\rho_{N_2}$, as  required.
\end{proof}

\begin{corollary}\label{C2} In any completely inverse $AG^{**}$-groupoid $A$, the relation
$$\pi=\sigma\cap\mu$$
is the least $E$-unitary congruence on $A$.
\end{corollary}

Observe that if $\rho$ is an $E$-unitary congruence on $A$, then ker$(\rho)=\text{ker}(\rho_N$) for some $N\lhd A$.~In the last section we will show that the converse implication is also true, that is, for any $AG$-group congruence $\rho_N$ on $A$ ($N\lhd A$), the family
$$\mathcal{U}_{N}=\{\rho_N\cap\upsilon:\mu\subseteq\upsilon\}$$
coincides with the set of all ($E$-unitary) congruence $\rho$ on $A$ such that
$$\text{ker}(\rho)=\text{ker}(\rho_N).$$

Finally, denote by $\mathcal{U}(A)$ the set of all $E$-unitary congruences  on a completely inverse $AG^{**}$-groupoid $A$. Since the intersection of an arbitrary nonempty family of $E$-unitary congruences on $A$ is again an $E$-unitary congruence on $A$, and  $\mathcal{U}(A)$ has a least element, then the following corollary is valid.

\begin{corollary}\label{C3} Let $A$ be a completely inverse $AG^{**}$-groupoid.~Then the set\hspace{0.3mm} $\mathcal{U}(A)$ is a complete $\cap$-sublattice of $\mathcal{C}(A)$ with the least element $\pi$ and the greatest element $A\times A$.

Moreover, $\mathcal{U}_{N}=\{\rho_N\cap\upsilon:\mu\subseteq\upsilon\}$ $(N\lhd A)$ is a complete sublattice of $\mathcal{U}(A)$ with the least element $\rho_N\cap\mu$ and the greatest element $\rho_N$.
\end{corollary}

In view of the corollary, for each $\rho\in\mathcal{C}(A)$, there is the least $E$-unitary congruence $\pi_{\rho}$ containing $\rho$. We will show in Section $6$ that $\pi_{\rho}=\sigma\rho\hspace{0.35mm}\sigma\cap\mu\rho\mu$.

\section{The trace classes of $\mathcal{C}(A)$}

Let $\rho$ be a congruence on $A$, where $A$ denotes (unless otherwise stated) an arbitrary completely inverse $AG^{**}$-groupoid.~Put $K=\text{ker}(\rho)$.~It is immediate that $K$ is a full completely inverse $AG^{**}$-subgroupoid of $A$.~In the light of Proposition \ref{ab=ba}, $K$ is also symmetric.~Finally, put $\rho_{(K,\tau)}=\rho$, where $\tau=\text{tr}(\rho)$.~Theorem \ref{kernel-trace} states that
\begin{equation}\label{ee4}
(a,b)\in\rho_{(K,\tau)}\iff (aa^{-1},bb^{-1})\in\tau~~\&~~ab^{-1}\in K.
\end{equation}
Notice that if $a\in\text{ker}(\rho_{(K,\tau)})$, that is, $(a,e)\in\rho_{(K,\tau)}$, where $e\in E_A$, then
$$
ea\in K~~\&~~(e,aa^{-1})\in\text{tr}(\rho_{(K,\tau)}).
$$
Observe further that if $ea\in K$ and $(e,aa^{-1})\in\text{tr}(\rho)$, then $\,a=(aa^{-1})a\,\rho\,ea$, therefore, $a\in K$.

Also, the following special case is of particular interest.

\begin{proposition}\label{U(A)} Let $A$ be a completely inverse $AG^{**}$-groupoid.~Then $\rho\in\mathcal{U}(A)$ if and only if $\,\ker(\rho)$ is closed in $A$.
\end{proposition}

\begin{proof} Let $\rho\in\mathcal{U}(A)$ and $a\in(\text{ker}(\rho))\omega$.~Then $b=ea$ for some $b\in\text{ker}(\rho)$ and $e\in E_A$. Hence $b\rho=(e\rho)(a\rho)$, where $e\rho,b\rho\in E_{A/\rho}$ and so $a\rho\in E_{A/\rho}$, since $A/\rho$ is $E$-unitary. Thus $a\in\text{ker}(\rho)$. Consequently, $(\text{ker}(\rho))\omega=\text{ker}(\rho)$.

Conversely, let $(e\rho)(a\rho)=f\rho$, where $a\in A$ and $e,f\in E_A$, then $ea\in\text{ker}(\rho)$. Hence $a\in(\text{ker}(\rho))\omega=\text{ker}(\rho)$, that is, $a\rho\in E_{A/\rho}$. Thus $\rho$ is $E$-unitary.
\end{proof}

In Section $3$ we have called a completely inverse $AG^{**}$-subgroupoid of $A$ normal if it is full, symmetric and closed in $A$.~Also, we say that a completely inverse $AG^{**}$-subgroupoid $K$ is \emph{seminormal} if $K$ is full and symmetric.

Finally, for any ordered pair $(K,\tau)$, where $K$ is a seminormal completely inverse $AG^{**}$-subgroupoid of $A$ and $\tau$ is a congruence on $E_A$ such that

\vspace{2mm}$(CP)$ \ if $ea\in K$ and $(e,aa^{-1})\in\tau$, then $a\in K$ ($a\in A,e\in E_A$),
\vspace{2mm}\newline define a relation $\rho_{(K,\tau)}$ like the above. In that case, $(K,\tau)$ is a \emph{congruence pair} for $A$ and we can define a relation $\rho_{(K,\tau)}$ as in \eqref{ee4} above.

The following theorem together with the above consideration and Theorem \ref{kernel-trace} says that any congruence on $A$  is of the form $\rho_{(K,\tau)}$, where $(K,\tau)$ is a congruence pair for $A$, and this expression is unique.

\begin{theorem}\label{ACP}If $(K,\tau)$ is a congruence pair for a completely inverse $AG^{**}$-groupoid $A$, then $\rho_{(K,\tau)}$ is the unique congruence on $A$ with $\ker(\rho_{(K,\tau)})=K$ and ${\rm tr}(\rho_{(K,\tau)})=\tau$.

Conversely, if $\rho$ is a congruence on $A$, then $(\ker(\rho),{\rm tr}(\rho))$ is a congruence pair for $A$ and $\rho_{(\ker(\rho),{\rm tr}(\rho))}=\rho$.
\end{theorem}

\begin{proof} It is sufficient to show the direct part of the theorem.~Put $\rho=\rho_{(K,\tau)}$. It is clear that $\rho$ is reflexive and symmetric.~Let now $(a,b),(b,c)\in\rho$.~Then $(aa^{-1},cc^{-1})\in\tau$ and $(b^{-1}a)(bc^{-1})=(b^{-1}b)(ac^{-1})=(bb^{-1})(ac^{-1})\in K$. Also,
$$bb^{-1}~\tau~(aa^{-1})(c^{-1}c)=(ac^{-1})(a^{-1}c)=(ac^{-1})(ac^{-1})^{-1}.$$ In the light of the condition $(CP)$, $ac^{-1}\in K$. Thus $\rho$ is transitive. Let $(a,b)\in\rho$ and $c\in A$. Then
$$(ca)(ca)^{-1}=(ca)(c^{-1}a^{-1})=(cc^{-1})(aa^{-1})~\tau~(cc^{-1})(bb^{-1})=(cb)(cb)^{-1},$$
$$(ac)(ac)^{-1}=(ac)(a^{-1}c^{-1})=(aa^{-1})(cc^{-1})~\tau~(bb^{-1})(cc^{-1})=(bc)(bc)^{-1}.$$
Also,
$$(ca)(cb)^{-1}=(ca)(c^{-1}b^{-1})=(cc^{-1})(ab^{-1})\in E_AK\subseteq KK\subseteq K,$$
$$(ac)(bc)^{-1}=(ac)(b^{-1}c^{-1})=(ab^{-1})(cc^{-1})\in KE_A\subseteq KK\subseteq K.$$
Consequently, $\rho$ is a congruence on $A$.

Finally, let $a\in\text{ker}(\rho)$, that is, $(a,e)\in\rho$ for some $e\in E_A$. Then clearly  $ea\in K$ and $(e,aa^{-1})\in\tau$. Hence $a\in K$ (by $(CP)$). Thus $\text{ker}(\rho)\subseteq K$. Conversely, let $a\in K$. Then $a^{-1}\in K$. Hence $(a^{-1}a,a)\in\rho$ and so $a\in\text{ker}(\rho)$. Thus ker$(\rho)=K$. Evidently, tr$(\rho)=\tau$. In view of Theorem \ref{kernel-trace}, $\rho_{(K,\tau)}$ is uniquely determined by the congruence pair $(K,\tau)$.
\end{proof}

It is easy to see that if $K$ is closed in $A$, then the condition $(CP)$ is not necessary in the proof of the direct part of Theorem \ref{ACP}.~Combining this fact with Proposition \ref{U(A)} and Theorem \ref{kernel-trace} we obtain the following corollary.

\begin{corollary}\label{C4} Each $E$-unitary congruence $\rho$ on a completely inverse $AG^{**}$-groupoid $A$ is of the form $\rho_{(K,\tau)}$, where $K\lhd A$ and $\tau\in\mathcal{C}(E_A)$, and this expression is unique.
\end{corollary}

\begin{remark} One can modify Proposition III.$2.3$ \cite{Pet} for completely inverse $AG^{**}$-groupoids.
\end{remark}

Further, let $\rho$ be a congruence on $A$. Put
$$\mu(\rho)=\{(a,b)\in A\times A:(a\rho,b\rho)\in\mu_{A/\rho}\}.$$
Clearly, $\mu(\rho)\in\mathcal{C}(A)$ and $\rho\subseteq\mu(\rho)$. From Theorem \ref{mu} follows that
$$(a,b)\in\mu(\rho)\iff (aa^{-1},bb^{-1})\in\rho.$$
Put $\mu(\rho)=\rho^{\theta}$. It is clear that $\text{tr}(\rho)=\text{tr}(\rho^\theta)$. Also, if $\text{tr}(\rho_1)=\text{tr}(\rho_2)$ ($\rho_1,\rho_2\in\mathcal{C}(A)$), then from the above equality follows that $\rho_1^{\theta}=\rho_2^{\theta}$. Consequently, $\rho^{\theta}$ is the maximum congruence with respect to tr$(\rho)$.
\newline Also, put (see Theorem \ref{sigma})
$$\rho_\theta=\{(a,b)\in A\times A:(aa^{-1},bb^{-1})\in\rho~~\&~~(\exists~e\in E_{(aa^{-1})\rho})~ea=eb\}.$$
Since $a=(aa^{-1})a$, then $\rho_\theta$ is reflexive.~Obviously, $\rho_\theta$ is symmetric.~The proof that $\rho_\theta$ is transitive and left compatible is closely similar to the corresponding proof for the relation $\sigma$ (see Theorem \ref{sigma}). Let $(a,b)\in\rho_\theta$ and $c\in A$. Then
$$(ac)(ac)^{-1}=(ac)(a^{-1}c^{-1})=(aa^{-1})(cc^{-1})~\rho~(bb^{-1})(cc^{-1})=(bc)(bc)^{-1}.$$
Also, $e(cc^{-1})~\rho~(aa^{-1})(cc^{-1})=(ac)(ac)^{-1}$ and
$$e(cc^{-1})\cdot ac = ea\cdot (cc^{-1})c= eb\cdot (cc^{-1})c=e(cc^{-1})\cdot bc.$$
Consequently, $\rho_\theta$ is a congruence on $A$.~Finally, from the definition of $\rho_\theta$ follows that tr$(\rho)=\text{tr}(\rho_\theta)$, and since the definition of $\rho_\theta$ depends only on idempotents, then $\rho_\theta$ is the minimum congruence with respect to tr$(\rho)$.

We have just proved part of the following theorem.

\begin{theorem}\label{main trace} Let $A$ be an arbitrary completely inverse $AG^{**}$-groupoid.~Define a map $\Theta:\mathcal{C}(A)\to\mathcal{C}(E_A)$ by
$$
\rho\hspace{0.1mm}\Theta={\rm tr}(\rho)~~(\rho\in\mathcal{C}(A)).
$$
Then $\Theta$ is a complete lattice homomorphism of $\mathcal{C}(A)$ onto $\mathcal{C}(E_A)$.~Also, if $\theta$ denotes the congruence on $\mathcal{C}(A)$ induced by $\Theta$, that is,
$$
\theta=\{(\rho_1,\rho_2)\in\mathcal{C}(A)\times\mathcal{C}(A):{\rm tr}(\rho_1)={\rm tr}(\rho_2)\},
$$
then for every $\rho\in\mathcal{C}(A)$,
$$\rho\theta=[\rho_\theta,\rho^\theta]$$
is a complete modular sublattice $($with commuting elements$)$ of $\,\mathcal{C}(A)$.
\end{theorem}

\begin{proof} The proof that $\Theta$ is a complete homomorphism is closely similar to the corresponding proof of Theorem III.2.5 \cite{Pet}, since the join of any nonempty family $\mathcal{F}$ of congruences in an arbitrary universal algebra is given by $\bigcup_{n\hspace{0.2mm}\in\hspace{0.35mm}\mathbb{N}}(\bigcup\mathcal{F})^n$. Further, let $\tau$ be a congruence on $E_A$. Define an equivalence relation $\rho$ on $A$ by
$$\rho=\{(a,b)\in A\times A:(aa^{-1},bb^{-1})\in\tau\}.$$
It is easy to check that $\rho$ is compatible with the operation on $A$.~Consequently, $\rho\in\mathcal{C}(A)$. Obviously, $\text{tr}(\rho)=\tau$. Thus $\Theta$ maps $\mathcal{C}(A)$ onto $\mathcal{C}(E_A)$.

Finally, $\rho\theta$ is an interval of a complete lattice, so it is itself \nolinebreak a complete lattice.~Let $\rho_1,\rho_2\in\rho\theta$ and $a(\rho_1\rho_2)b$.~Then $a\rho_1 c\rho_2b$, where $c\in A$, so $(aa^{-1})\rho_1(cc^{-1})\rho_2(bb^{-1})$. Hence $(aa^{-1})\rho_2(cc^{-1})\rho_1(bb^{-1})$, since ${\rm tr}(\rho_1)={\rm tr}(\rho_2)$.~Moreover, $(cc^{-1})\rho_2(bc^{-1})$. It follows that $(aa^{-1})\rho_2(bc^{-1})$. Consequently,
$$a=(aa^{-1}\cdot a)\rho_2(bc^{-1}\cdot a)=(ac^{-1})b.$$
Further, $(ac^{-1})\rho_1(cc^{-1})$ and so $(ac^{-1})\rho_1(bb^{-1})$. Hence $(ac^{-1}\cdot b)\rho_1(bb^{-1}\cdot b)=b$. We have just shown that $a\rho_2(ac^{-1}\cdot b)\rho_1b$, that is, $\rho_1\rho_2\subseteq\rho_2\rho_1$.~By symmetry, we deduce that $\rho_1\rho_2=\rho_2\rho_1$, therefore, the lattice $\rho\theta$ is modular.
\end{proof}

We call the classes of $\theta$ in the above theorem, the \emph{trace classes} of $A$.

\begin{lemma}\label{L1}  Let $A$ be a completely inverse $AG^{**}$-groupoid.~Then
$$
\rho_\theta\subseteq\gamma_\theta\iff {\rm tr}(\rho)\subseteq{\rm tr}(\gamma)\iff\rho^\theta\subseteq\gamma^\theta$$
for all $\rho,\gamma\in\mathcal{C}(A)$. Also, if $\rho\subseteq\gamma$, then $\rho_\theta\subseteq\gamma_\theta$ and $\rho^\theta\subseteq\gamma^\theta$.
\end{lemma}

\begin{proof}This follows directly from the definitions of $\rho_\theta$ and $\rho^\theta$.
\end{proof}

\begin{lemma}\label{L2}  Let $\mathcal{F}$ be an arbitrary nonempty family of congruences on a completely inverse $AG^{**}$-groupoid. Put
$$\mathcal{F}_\theta=\{\rho_\theta:\rho\in\mathcal{F}\},~~\mathcal{F}\hspace{0.4mm}^\theta=\{\rho^\theta:\rho\in\mathcal{F}\}.$$
Then
$$
\bigvee\mathcal{F}_\theta=\Big(\bigvee\mathcal{F}\Big)_\theta~~\&~~\bigcap\mathcal{F}\hspace{0.4mm}^\theta=\Big(\bigcap\mathcal{F}\Big)^\theta.
$$
\end{lemma}

\begin{proof}The proof is similar to the proof of Lemma III.2.9 \cite{Pet}.
\end{proof}

\begin{lemma}\label{L3}  Let $A$ a completely inverse $AG^{**}$-groupoid.~Then $\sigma=(A\times A)_\theta$.
\end{lemma}

\begin{proof} This is obvious.
\end{proof}

The following corollary gives another equivalent conditions for a completely inverse $AG^{**}$-groupoid to be $E$-unitary.

\begin{corollary}\label{E-unitary3} Let $A$ be a completely inverse $AG^{**}$-groupoid.~The following conditions are equivalent$:$

$(a)$ \ $A$ is $E$-unitary$;$

$(b)$ \ $\rho_\theta=\rho\cap\sigma$ for every $\rho\in\mathcal{C}(A);$

$(c)$ \ $\rho_\theta$ is an idempotent pure congruence on $A$ for every $\rho\in\mathcal{C}(A)$.
\end{corollary}

\begin{proof} Recall that $A$ is $E$-unitary if and only if $\sigma$ is the maximum idempotent pure congruence on $A$ (Theorem \ref{E-unitary2}).

$(a)\implies (b)$. If $\rho\in\mathcal{C}(A)$, then $\rho_\theta\subseteq\rho\cap(A\times A)_\theta=\rho\cap\sigma$ (Lemmas \ref{L1}, \ref{L3}). On the other hand,
$$\text{tr}(\rho\cap\sigma)=\text{tr}(\rho)\cap\text{tr}(\sigma)=\text{tr}(\rho)\cap (E_A\times E_A)=\text{tr}(\rho)=\text{tr}(\rho_\theta)$$
and
$$\text{ker}(\rho\cap\sigma)=\text{ker}(\rho)\cap\text{ker}(\sigma)=\text{ker}(\rho)\cap E_A=E_A\subseteq\text{ker}(\rho_\theta).$$
Thus $\rho\cap\sigma\subseteq\rho_\theta$ (Theorem \ref{kernel-trace}). Consequently,  $\rho_\theta=\rho\cap\sigma$.

$(b)\implies (a)$. Clearly, $\mu_\theta=1_A$.~Moreover, $\mu_\theta=\mu\cap\sigma=\pi$ (Corollary \ref{C2}), therefore, $\pi=1_A$, so $A$ is $E$-unitary.

It is now clear that $(a)$ implies $(c)$. We show the opposite implication. Indeed, if $(c)$ holds, then $(A\times A)_\theta=\sigma$ is idempotent pure. Since $\rho_\theta\subseteq\sigma$ for every $\rho\in\mathcal{C}(A)$, then each $\rho_\theta$ is idempotent pure, too, as required.
\end{proof}

We have mentioned in the above proof that if $A$ is $E$-unitary, then $\sigma$ is the maximum idempotent pure congruence on $A$, therefore, the set of all idempotent pure congruences $[1_A,\sigma]$ on an $E$-unitary completely inverse $AG^{**}$-groupoid $A$ forms a complete sublattice of the lattice $\mathcal{C}(A)$.

From the above corollary we obtain the following proposition.

\begin{proposition}\label{i-p} Let $A$ be an $E$-unitary completely inverse $AG^{**}$-groupoid.~Then the mapping $\chi:\mathcal{C}(A)\to \mathcal{C}(A)$ defined by
$$\rho\chi=\rho\cap\sigma~~(\rho\in\mathcal{C}(A))$$
is a complete lattice homomorphism of $\mathcal{C}(A)$ onto the lattice of all idempotent pure congruences on $A$.
\end{proposition}

\begin{proof} In view of Corollary \ref{E-unitary3},  $\rho_\theta=\rho\cap\sigma$ for every $\rho\in\mathcal{C}(A)$.~Hence $\chi$ is a complete $\vee$-homomorphism (by Lemma \ref{L2}).~It is evident that $\chi$ is a complete $\cap$-homomorphism. Finally, if $\rho$ is idempotent pure, then $\rho\subseteq\sigma$ and so $\rho\chi=\rho$. Thus $\chi$ maps $\mathcal{C}(A)$  onto the lattice of all idempotent pure congruences on $A$, as exactly required.
\end{proof}

We now investigate the $\theta$-classes of $A$.

\begin{lemma}\label{pomoc} In any completely inverse $AG^{**}$-groupoid $A$, $\mu_{A/\rho}=\mu(\rho)/\rho$ for every $\rho\in\mathcal{C}(A)$.~In particular, $[\rho/\rho,\rho^\theta/\rho]$ is the modular lattice of all idempotent-separating congruences on $A/\rho$ $(\rho\in\mathcal{C}(A))$.
\end{lemma}

\begin{proof} It is easy to see that $\mu(\rho)/\rho$ is idempotent-separating, so $\mu(\rho)/\rho\subseteq\mu_{A/\rho}$. On the other hand, if $\gamma/\rho$, where $\rho\subseteq\gamma$, is an idempotent-separating congruence on $A/\rho$, then $\text{tr}(\gamma)\subseteq\text{tr}(\rho)$ and so $\text{tr}(\gamma)=\text{tr}(\rho)$. Hence $\rho\subseteq\gamma\subseteq\mu(\rho)$, therefore, $\gamma/\rho\subseteq\mu(\rho)/\rho$.~Thus $\mu_{A/\rho}=\mu(\rho)/\rho$.~The second part of the lemma follows from Theorem \ref{m}.
\end{proof}

The following theorem follows easily from the above lemma.

\begin{theorem}\label{trace class} Let $A$ be a completely inverse $AG^{**}$-groupoid, $\rho\in \mathcal{C}(A)$.~Define a map $\phi : [\rho_\theta, \rho^\theta] \to A/\rho_\theta$ by $\rho\phi = \rho/\rho_\theta$ for all $\rho\in [\rho_\theta, \rho^\theta]$.~Then $\phi$ is a complete isomorphism of the trace class $[\rho_\theta, \rho^\theta]$ onto the modular lattice of all idempotent-separating congruences on $A/\rho_\theta$.
\end{theorem}

\begin{remark} Note that $\phi_{\hspace{0.4mm}|\hspace{0.4mm}[\gamma,\hspace{0.3mm}\mu(\rho)]}$, where $\gamma\in\rho\theta$, is a complete isomorphism of the interval $[\gamma, \mu(\rho)]$ onto the lattice of all idempotent-separating congruences on $A/\gamma$.
\end{remark}

Recall that $A$ is \emph{fundamental} if and only if $\mu = 1_A$. By the above remark we have the following corollary.

\begin{corollary}\label{fundamental} Let $\rho$ be a congruence on  a completely inverse $AG^{**}$-groupoid \nolinebreak $A$. Then $A/\rho$ is fundamental if and only if $\rho=\mu(\rho)$.
\end{corollary}

Denote by $\mathcal{FC}(A)$ the set of all fundamental congruences on $A$, that is,
$$\mathcal{FC}(A)=\{\mu(\rho):\rho\in\mathcal{C}(A)\}.$$

Since $1_A\subseteq\rho$, then $\mu=\mu(1_A)\subseteq\mu(\rho)$ for all $\rho\in\mathcal{C}(A)$, what means that $\mu$ is the least fundamental congruence on $A$. Also, from Lemma \ref{L2} follows that $\mathcal{FC}(A)$ is a complete $\cap$-sublattice of $\mathcal{C}(A)$.

We have just proved a part of the following theorem.

\begin{theorem} Let $A$ be a completely inverse $AG^{**}$-groupoid.~Then $\mathcal{FC}(A)$ is a complete $\cap$-sublattice of $\mathcal{C}(A)$ with the least element $\mu$ and the greatest element $A\times A$. For any nonempty family $\{\rho_i: i\in I\}$ of fundamental congruences on $A$, the join of $\{\rho_i: i\in I\}$ in $\mathcal{FC}(A)$ is given by $\mu (\bigvee\{\rho_i: i\in I\})$.~Also, $\mathcal{FC}(A)\cong\mathcal{C}(E_A)$.
\end{theorem}

\begin{proof} Let $\emptyset\not=\{\rho_i: i\in I\}\subseteq\mathcal{FC}(A)$. Then
$$\Big(\bigvee\{\rho_i: i\in I\}, \mu\Big(\bigvee \{\rho_i: i\in I\}\Big)\Big)\in\theta.$$
On the other hand, if $\rho\in\Big[\bigvee\{\rho_i: i\in I\}, \mu(\bigvee \{\rho_i: i\in I\})\Big]$, then $\mu(\rho)=\rho$ if and only if $\rho=\mu(\bigvee \{\rho_i: i\in I\})$.~Consequently, $\mu(\bigvee \{\rho_i: i\in I\})$ is the join of $\{\rho_i: i\in I\}$ in $\mathcal{F}(S)$.

Finally, if $\mu(\rho_1)\neq\mu(\rho_2)$, where $\rho_1,\rho_2\in\mathcal{C}(A)$, then $\text{tr}(\rho_1)\neq\text{tr}(\rho_2)$, therefore, the restriction of the map $\Theta$ from Theorem \ref{main trace} to the set $\mathcal{FC}(A)$ is the required complete lattice isomorphism.
\end{proof}

\section{The kernel classes of $\mathcal{C}(A)$}

Let $A$ be an $AG^{**}$-groupoid. For every nonempty subset $Q$ of $A$ there exists an associated  equivalence relation $\mathcal{Q}$ on $A$ which is induced by the partition: $\{Q,A\setminus Q\}$. Define on $A$ an equivalence relation $\tau^Q$ by
$$
\tau^Q=\{(a,b)\in A\times A:(\forall x,y\in A^{1})~x(ay)\in Q\iff x(by)\in Q\},
$$
where $A^{1}=A\cup\{1\}$, $1\not\in A$ and $1a=a1=a$ for all $a\in A$.

Observe that if $(a,b)\in\tau^Q$, then putting $x=y=1$ in the definition of $\tau^Q$, we obtain that either $a,b\in Q$ or $a,b\notin Q$. Thus $\tau^Q\subseteq\mathcal{Q}$.

\begin{proposition}\label{saturates} Let $Q$ be a nonempty subset of an $AG^{**}$-groupoid $A$.~Then $\tau^Q$ is the largest congruence $\rho$ on $A$ for which $Q$ is the union of some $\rho$-classes.
\end{proposition}

\begin{proof} Let $(a,b)\in\tau^Q,x,y\in A^1$ and $c\in A$. Observe that
$$
x(ac\cdot y)=(ac)(xy)=(xy\cdot c)a.
$$
Hence if  $x(ac\cdot y)\in Q$, then $(xy\cdot c)b\in Q$, since $(a,b)\in\tau^Q$. Thus we get $x(bc\cdot y)\in Q$. By symmetry, we conclude that  $\tau^Q$ is right compatible. Further, the equality
$$
x(ca\cdot y)=(ca)(xy)=(cx)(ay)
$$
implies that $\tau^Q$ is also left compatible.~Consequently, $\tau^Q$ is a congruence on $A$ and $Q$ is the union of some $\tau^Q$-classes, since $\tau^Q\subseteq\mathcal{Q}$.~Finally, if $\rho$ is any congruence on $A$ for which $Q$ is the union of some $\rho$-classes, then $\rho\subseteq\mathcal{Q}$. Hence if $(a,b)\in\rho$, then either $a,b\in Q$ or $a,b\notin Q$.~Thus for all $x,y\in A^1$, $x(ay)\in Q\Leftrightarrow x(by)\in Q$, so $a~\tau^Q~b$.~Consequently, $\rho\subseteq\tau^Q$.
 \end{proof}

\begin{corollary} In any completely inverse $AG^{**}$-groupoid $A$, the relation $\tau\hspace{0.15mm}^{E_A}$ is the largest idempotent pure congruence on $A$.
\end{corollary}

We shall write $\tau$ instead of $\tau\hspace{0.15mm}^{E_A}$, or $\tau_A$ if necessary.

Let $\rho$ be a congruence on $A$, where $A$ denotes (unless otherwise stated) an arbitrary completely inverse $AG^{**}$-groupoid. Put
$$
\tau(\rho)=\{(a,b)\in A\times A:(a\rho,b\rho)\in\tau_{A/\rho}\}.
$$
Clearly, $\tau(\rho)\in\mathcal{C}(A)$ and $\rho\subseteq\tau(\rho)$. Using Theorem \ref{i-s}, one can prove without difficulty that $\tau(\rho)=\tau\hspace{0.2mm}^{\text{ker}(\rho)}$. Thus $\tau(\rho)$ is the maximum congruence with respect to ker$(\rho)$. Denote it by $\rho^\kappa$.

Further, put $\rho_\kappa=\rho\cap\mu$. Then $\text{ker}(\rho_\kappa)=\text{ker}(\rho)$, since $\mu$ is a semilattice congruence. On the other hand, $\mu$ is idempotent-separating, so $\rho_\kappa$ is the minimum congruence with respect to ker$(\rho)$.

Finally if $K$ is a seminormal completely inverse $AG^{**}$-subgroupoid of $A$. Then the pair $(K,1_{E_A})$ is a congruence pair for $A$, since then the condition $(CP)$ is trivially met for this pair, and $\text{ker}(\rho_{(K,1_{E_A})})=K$.~Consequently, $K$ is seminormal if and only if $K$ is a kernel of some congruence on $A$.~Denote by $\mathcal{SN}(A)$ the set of seminormal completely inverse $AG^{**}$-subgroupoids of $A$. It is easy to see that $\mathcal{SN}(A)$ is a lattice under inclusion.

It is clear that if $\emptyset\not=\{\rho_i: i\in I\}\subseteq\mathcal{C}(A)$, then
$$
\text{ker}\Big(\bigcap\{\rho_i:i\in I\}\Big)=\bigcap\{\text{ker}(\rho_i):i\in I\},
$$
therefore, we have just proved the following theorem.

\begin{theorem}\label{main kernel} Let $A$ be an arbitrary completely inverse $AG^{**}$-groupoid.~Define a map $K:\mathcal{C}(A)\to\mathcal{P}(A)$ by
$$
\rho\hspace{0.1mm}K=\text{ker}(\rho)~~(\rho\in\mathcal{C}(A)).
$$
Then $K$ is a complete lattice $\cap$-homomorphism of $\mathcal{C}(A)$ onto $\mathcal{SN}(A)$.~Also, if $\kappa$ denotes the $\cap$-congruence on $\mathcal{C}(A)$ induced by $K$, that is,
$$
\kappa=\{(\rho_1,\rho_2)\in\mathcal{C}(A)\times\mathcal{C}(A):\ker(\rho_1)=\ker(\rho_2)\},
$$
then for ever $\rho\in\mathcal{C}(A)$,
$$\rho\kappa=[\rho_\kappa,\rho^\kappa]$$
is a complete sublattice of $\,\mathcal{C}(A)$.
\end{theorem}

We call the classes of $\kappa$ in the above theorem, the \emph{kernel classes} of $A$.

\begin{example}\label{counterex} The following example shows that ker$(\rho) \subseteq \text{ker}(\gamma)$ (or even $\rho\subseteq \gamma$) does not imply (in general) that $\rho^\kappa \subseteq \gamma^\kappa$.~\hspace{0.3mm}Indeed, let $A = \{a, b, e, f\}$ be a commutative inverse semigroup with the multiplication table given below:

$$
\begin{array}{|c||c|c|c|c|}
\hline
\cdot & a & b & e & f\\
\hline
\hline
a & e & e & a & a\\
\hline
b & e & f & a & b\\
\hline
e & a & a & e & e\\
\hline
f & a & b & e & f\\
\hline
\end{array}
$$
Then clearly $1_A \subseteq \rho = 1_A \cup \{(a, e), (e, a)\}$. On the other hand, $\rho^\kappa = \rho \cup \{(b, f), (f, b)\}$ and $1_A^\kappa = 1_A \cup \{(e, f), (f, e), (a, b), (b, a)\}$ and so  $1_A^\kappa = \tau \nsubseteq \rho^\kappa = \tau (\rho)$. Notice also that $\tau \cap \tau (\rho) = 1_A$.
\end{example}

Using Theorem \ref{kernel-trace} one can easily prove the following proposition.

\begin{proposition}\label{decomposition}If $\rho$ is a congruence  on a completely inverse $AG^{**}$-groupoid, then
$$\rho=\rho_\theta\vee\rho_\kappa=\rho^\theta\cap\rho^\kappa.$$
\end{proposition}

We now investigate the $\kappa$-classes of $A$.

\begin{lemma}\label{pomoc1} In any completely inverse $AG^{**}$-groupoid $A$, $\tau_{A/\rho}=\tau(\rho)/\rho$ for every $\rho\in\mathcal{C}(A)$.~In particular, $[\rho/\rho,\rho^\kappa/\rho]$ is the  lattice of all idempotent pure congruences on $A/\rho$ $(\rho\in\mathcal{C}(A))$.
\end{lemma}

\begin{proof} One can easily see that $\tau(\rho)/\rho$ is idempotent pure and so $\tau(\rho)/\rho\subseteq\tau_{A/\rho}$. On the other hand, if $\gamma/\rho$, where $\rho\subseteq\gamma$, is idempotent pure, then $\text{ker}(\gamma)\subseteq\text{ker}(\rho)$, therefore, $\text{ker}(\gamma)=\text{ker}(\rho)$.~Hence $\rho\subseteq\gamma\subseteq\tau(\rho)$, so $\gamma/\rho\subseteq\tau(\rho)/\rho$.~Consequently, $\tau_{A/\rho}=\tau(\rho)/\rho$.
\end{proof}

From the above lemma follows the following theorem.

\begin{theorem}\label{kernel class} Let $A$ be a completely inverse $AG^{**}$-groupoid, $\rho\in \mathcal{C}(A)$.~Define a map $\phi : [\rho_\kappa, \rho^\kappa] \to A/\rho_\kappa$ by $\rho\phi = \rho/\rho_\kappa$ for all $\rho\in [\rho_\kappa, \rho^\kappa]$.~Then $\phi$ is a complete isomorphism of the kernel class $[\rho_\kappa, \rho^\kappa]$ onto the lattice of all idempotent pure congruences on $A/\rho_\kappa$.
\end{theorem}

Note that $\phi_{\hspace{0.4mm}|\hspace{0.4mm}[\gamma,\hspace{0.3mm}\tau(\rho)]}$, where $\gamma\in\rho\kappa$, is a complete isomorphism of the interval $[\gamma, \tau(\rho)]$ onto the lattice of all idempotent pure congruences on $A/\gamma$.

Recall that $A$ is \emph{$E$-disjunctive} if and only if $\tau = 1_A$. By the above remark we have the following corollary.

\begin{corollary}\label{E-disjunctive} Let $\rho$ be a congruence on  a completely inverse $AG^{**}$-groupoid \nolinebreak $A$. Then $A/\rho$ is $E$-disjunctive if and only if $\rho=\tau(\rho)$.
\end{corollary}

\begin{remark} Note that in view of the end of Example \ref{counterex}, the set of all $E$-disjunctive congruences on a commutative inverse semigroup (in particular, on a completely inverse $AG^{**}$-groupoid) $A$ does not form (in general) a sublattice of $\mathcal{C}(A)$.

Also, a completely inverse $AG^{**}$-groupoid $A$ is an $AG$-group if and only if $A$ is both $E$-unitary and $E$-disjunctive.

Finally, notice that a congruence $\rho$ on $A$ is idempotent pure if and only if $\rho\cap\mu=1_A$. In particular, $\tau\cap\mu=1_A$, therefore, $A$ is a subdirect product of $A/\tau$ and $E_A$, where $A/\tau$ is an $E$-disjunctive completely inverse $AG^{**}$-groupoid.
\end{remark}

\medskip
Finally, we go back to study the lattice $\mathcal{U}(A)$ of all $E$-unitary congruences on a completely inverse $AG^{**}$-groupoid $A$. First, we prove the following useful result.

\begin{lemma}\label{lemma} The following conditions are valid for a congruence $\rho$ on a completely inverse $AG^{**}$-groupoid $A$\emph{:}

$(a)$ \ $\rho\vee\sigma=\sigma\rho\hspace{0.3mm}\sigma;$

$(b)$ \ $a(\rho\vee\sigma)b\Leftrightarrow (ea)\rho(eb)$ for some $e\in E_A;$

$(c)$ \ $\ker(\rho\vee\sigma)=(\ker(\rho))\omega$.
\end{lemma}

\begin{proof} Using Proposition \ref{semilattice}, we may show, in a very similar way like in the proof of Lemma III.5.4$(i)$ \cite{Pet}, the condition $(a)$.~Furthermore, the condition $(b)$ follows directly from Proposition \ref{semilattice} and $(a)$.~Finally, the proof of $(c)$ is closely similar to the corresponding proof of Corollary III.5.5 \cite{Pet}.
\end{proof}

Using Proposition \ref{semilattice} and Lemma \ref{lemma}$(b)$, we are able to show the following theorem.

\begin{theorem}\label{onto} Let $A$ be an arbitrary completely inverse $AG^{**}$-groupoid.~Then the map $\phi:\mathcal{C}(A)\to\mathcal{C}(A)$ defined by
$$
\rho\phi=\rho\vee\sigma
$$
is a homomorphism of $\mathcal{C}(A)$ onto the lattice $[\sigma,A\times A]$ of all $AG$-group congruences on $A$.
\end{theorem}

Define  the relation $\bar{\sigma}$ on $\mathcal{C}(A)$ by putting
$$
(\rho_1, \rho_2) \in \bar{\sigma}\Leftrightarrow\rho_1\vee\sigma = \rho_2 \vee \sigma .
$$
 In the light of the above theorem, $\bar{\sigma}$ is a congruence on $\mathcal{C}(A)$, since $\phi \phi^{- 1} = \bar{\sigma}$.

\begin{proposition} Let $A$ be a completely inverse $AG^{**}$-groupoid and $\rho \in \mathcal{C}(A)$. Then the elements $\rho, \pi_{\rho}$ and $\rho \vee \sigma$ are $\bar{\sigma}$-equivalent and $\rho \subseteq \pi_{\rho} \subseteq \rho \vee \sigma$.~Moreover, the element $\rho \vee \sigma$ is the largest in the $\bar{\sigma}$-class $\rho \bar{\sigma}$.
\end{proposition}

\begin{proof} Since  $\pi_{\rho}$ is the least $E$-unitary congruence containing $\rho$ and $\rho \vee \sigma$  is $E$-unitary, then $\rho \subseteq \pi_{\rho} \subseteq \rho \vee \sigma$. Hence we get $\rho \vee \sigma \subseteq \pi_{\rho} \vee \sigma\subseteq \rho \vee \sigma$, so $\rho \vee \sigma = \pi_{\rho} \vee \sigma$, therefore, $(\rho, \pi_{\rho}) \in \bar{\sigma}$. Evidently, $(\rho, \rho \vee \sigma) \in \bar{\sigma}$.~This implies the first part of the proposition. The second part is clear.
\end{proof}

Further, let $a,b\in A$ and $\rho\in\mathcal{C}(A)$. If $(a,b)\in\sigma$, then evidently $(a\rho)\sigma(b\rho)$ in $S/\rho$. If in addition, $\rho\subseteq\sigma$, then $(a\rho)\sigma(b\rho)$ in $S/\rho$ implies that $(a,b)\in\sigma$ in $S$.~It follows that $A/\sigma\cong (A/\rho)/\sigma$, i.e., $A$ and $A/\rho$ have isomorphic maximal $AG$-group homomorphic images. In that case, we may say that $\rho$ \emph{preserves the maximal $AG$-group homomorphic images}.~Since for every $\rho\in\mathcal{C}(A)$ we have $\rho_\theta\subseteq\rho$, then we obtain the following factorization:
$$A\to A/\rho_\theta\to A/\rho \cong (A/\rho_\theta)/(\rho/\rho_\theta).$$

Using the obvious terminology, we have the following proposition.

\begin{proposition} Every homomorphism of completely inverse $AG^{**}$-groupoids can be factored into a homomorphism preserving the maximal $AG$-group homomorphic images and an idempotent-separating homomorphism.
\end{proposition}

\begin{proof} The proof is similar to the proof of Proposition III.5.10 \cite{Pet}.
\end{proof}

The following theorem gives another equivalent conditions for a congruence to be $E$-unitary (cf.~the end of Section $4$).

\begin{theorem}\label{ME} Let $\rho$ be a congruence on a completely inverse $AG^{**}$-groupoid $A$.~Then the following conditions are equivalent\emph{:}
\par {\indent\rm$(a)$} \ $\rho$ is $E$-unitary\emph{;}
\par {\indent\rm$(b)$} \ $\ker(\rho)$ is closed\emph{;}
\par {\indent\rm$(c)$} \ $\ker(\rho) =\ker(\rho\vee\sigma);$
\par {\indent\rm$(d)$} \ $\rho\vee\sigma=\tau(\rho);$
\par {\indent\rm$(e)$} \ $\tau(\rho)\in\mathcal{GC}(A)$.
\end{theorem}

\begin{proof} In the light of Proposition \ref{U(A)}, $(a)$ and $(b)$ are equivalent.

$(b)\implies (c)$. This follows from Lemma \ref{lemma}$(c)$.

$(c)\implies (d)$. Indeed, $\text{ker}(\tau(\rho))=\text{ker}(\rho)=\text{ker}(\rho\lor\sigma)$ and so $\rho\vee\sigma\subseteq\tau(\rho)$, therefore, $\tau(\rho)\in\mathcal{GC}(A)$.

$(d)\implies (a)$. Let $\tau(\rho)\in\mathcal{GC}(A)$. Then $\tau(\rho)$ is $E$-unitary. Since the conditions $(a)$ and $(b)$ are equivalent, we get $\text{ker}(\rho)=\text{ker}(\tau(\rho))$ is closed. Thus $\rho$ is $E$-unitary.

$(c)\implies (d)$. By the above $\rho\lor\sigma\subseteq\tau(\rho)$ and so $\rho\lor\sigma, \tau(\rho)\in\mathcal{GC}(S)$. Further\-more, $\text{ker}(\tau(\rho))=\text{ker}(\rho)=\text{ker}(\rho\lor\sigma)$. Hence $\rho\lor\sigma=\tau(\rho)$ (by Theorem \ref{AG-group congruences}).

$(d)\implies (e)$. This is trivial.
\end{proof}

In view of the above theorem, Theorem \ref{main kernel} and Corollary \ref{C3},
$$\rho_N\kappa=\{\rho_N\cap\nu:\mu\subseteq\nu\}=[\rho_N\cap\mu,\rho_N]$$
for every $N\lhd A$. Consequently,
$$
\mathcal{U}(A)=\bigcup_{N\hspace{0.2mm}\lhd\hspace{0.35mm}A}\{\rho_N\cap\nu:\mu\subseteq\nu\}.
$$
Thus we have the following statement (see the end of Section $4$).

\begin{proposition}\label{pi_rho} Let $\rho$ be a congruence on a completely inverse $AG^{**}$-groupoid \nolinebreak $A$. Then$:$

$(a)$ \ $\rho\vee\mu=\mu\rho\mu;$

$(b)$ \ $a(\rho\vee\mu)b\Leftrightarrow (aa^{-1})\rho(bb^{-1});$

$(c)$ \ $\pi_\rho=\sigma\rho\hspace{0.3mm}\sigma\cap\mu\rho\mu$.
\end{proposition}

\begin{proof} $(a)$. It is clear that $\mu\rho\mu\subseteq\rho\vee\mu$ is a reflexive, symmetric and compatible relation on $A$. We show that it is also transitive. Let $a(\mu\rho\mu)b(\mu\rho\mu)c$. Then there exist elements $r,s,t,w\in A$ such that
$$aa^{-1}=rr^{-1},\ \ (r,s)\in\rho, \ \ ss^{-1}=bb^{-1},$$
$$bb^{-1}=tt^{-1}, \ \ (t,w)\in\rho, \ \ ww^{-1}=cc^{-1}.$$
Also, $(rr^{-1})\rho(ss^{-1})=(tt^{-1})\rho(ww^{-1})$. Consequently,
$$a~\mu(aa^{-1})=(rr^{-1})\rho(ww^{-1})=(cc^{-1})\mu~c.$$
Hence $(a,c)\in\mu\rho\mu$, as required, and so $\mu\rho\mu$ is a congruence on $A$ contained in $\rho\vee\mu$. Since evidently $\rho,\mu\subseteq\mu\rho\mu$, then $(a)$ holds.

$(b)$. $(\Longrightarrow)$. Let $a(\rho\hspace{0.35mm}\vee\hspace{0.2mm}\mu)b$. Then by $(a)$, $aa^{-1}=cc^{-1},(c,d)\in\rho$ and $dd^{-1}=bb^{-1}$ for some $c,d\in A$. Hence $(cc^{-1})\rho(dd^{-1})$. Thus $(aa^{-1})\rho(bb^{-1})$.

$(\Longleftarrow)$. If $(aa^{-1})\rho(bb^{-1})$, then $a~\mu(aa^{-1})\rho(bb^{-1})\mu~b$. Thus $a(\rho\vee\mu)b$.

$(c)$. In the light of Lemma \ref{lemma} $(a)$ and the condition $(b)$, $\alpha=\sigma\rho\hspace{0.3mm}\sigma\cap\mu\rho\mu$ is a congruence on $A$. It is evident that $\rho\subseteq\alpha$ and ker$(\alpha)=$ ker$(\sigma\rho\hspace{0.3mm}\sigma)$, therefore, $\alpha$ is an $E$-unitary congruence on $A$ which contains $\rho$. Finally, let $\rho\subseteq\beta=\rho_N\cap\nu\in\mathcal{U}(A)$, where $N\lhd A$ and $\mu\subseteq\nu$. Then $\rho\vee\sigma\subseteq\rho_N\vee\sigma=\rho_N$ and $\rho\vee\mu\subseteq\nu\vee\mu=\nu$. It follows that $\alpha\subseteq\rho_N\cap\nu=\beta$, as required.
\end{proof}

Using the condition $(b)$ one can prove the following theorem.

\begin{theorem}Let $A$ be an arbitrary completely inverse $AG^{**}$-groupoid.~Then the map $\phi:\mathcal{C}(A)\to\mathcal{C}(A)$ defined by
$$\rho\phi=\rho\vee\mu$$
is a homomorphism of $\,\mathcal{C}(A)$ onto the lattice $[\mu,A\times A]$ of semilattice congruences on $A$.
\end{theorem}

Let $\rho\in\mathcal{C}(A)$.~Since $\rho\subseteq A\times A$, then there is the least semilattice congruence $\mu_\rho$ containing $\rho$ (note that $\mu_\rho=\pi_\rho$, see the proof of Proposition \ref{pi_rho}$(c)$).

\medskip
Define the relation $\overline{\mu}$ on $\mathcal{C}(A)$ by putting
$$
(\rho_1, \rho_2) \in \overline{\mu}\Leftrightarrow\rho_1\vee\mu = \rho_2 \vee \mu.
$$
In view of the above theorem, $\overline{\mu}$ is a congruence on $\mathcal{C}(A)$.

\begin{proposition} Let $A$ be a completely inverse $AG^{**}$-groupoid, $\rho \in \mathcal{C}(A)$. Then the elements $\rho, \mu_{\rho}$ and $\rho \vee \mu$ are $\overline{\mu}$-equivalent and $\rho \subseteq \mu_{\rho} \subseteq \rho \vee \mu$.~Moreover, the element $\rho \vee \mu$ is the largest in the $\overline{\mu}$-class $\rho \overline{\mu}$.
\end{proposition}

Also, let $a,b\in A$ and $\rho\in\mathcal{C}(A)$. If $(a,b)\in\mu$, then clearly $(a\rho)\mu(b\rho)$ in \nolinebreak $S/\rho$. If in addition, $\rho\subseteq\mu$, then $(a\rho)\mu(b\rho)$ in $S/\rho$ implies that $(a,b)\in\mu$, since $\mu$ is idempotent-separating.~It follows that $A/\mu\cong (A/\rho)/\mu$, that is, $A$ and $A/\rho$ have isomorphic minimal idempotent-separating homomorphic images.~We may say that $\rho$ \emph{preserves the minimal idempotent-separating homomorphic images}.~Since for all $\rho\in\mathcal{C}(A)$, $\rho_\kappa\subseteq\rho$, then we have the following factorization:
$$
A\to A/\rho_\kappa\to A/\rho \cong (A/\rho_\kappa)/(\rho/\rho_\kappa).
$$

We get the following proposition.

\begin{proposition} Every homomorphism of completely inverse $AG^{**}$-group\-oids can be factored into a homomorphism preserving the minimal idempotent-separa\-ting homomorphic images and an idempotent pure homomorphism.
\end{proposition}

\begin{proof}Let $\rho$ be a congruence on a completely inverse $AG^{**}$-groupoid $A$. Then obviously $\rho_\kappa\subseteq\mu$, and hence the canonical homomorphism of $A$ onto $A/\rho_\kappa$ preserves the minimal idempotent-separating homomorphic images. Also, the mapping $a\rho_\kappa\to a\rho$  $(a\in A)$ is an idempotent pure homomorphism of $A/\rho_\kappa$ onto $A/\rho$, since $\text{ker}(\rho)=\text{ker}(\rho_\kappa)$.~The thesis of the proposition follows now from the above factorization.
\end{proof}

Since $\mu$ is also the least semilattice congruence on $A$ (Theorem \ref{mu}), then we may replace in the above proposition the words "minimal idempotent-separating" by the words "maximal semilattice".

Once again we prove some equivalent conditions for $A$ to be $E$-unitary.

\begin{theorem} Let $A$ be a completely inverse $AG^{**}$-groupoid.~The following conditions are equivalent$:$
\par {\indent\rm$(a)$} \ $A$ is $E$-unitary$;$
\par {\indent\rm$(b)$} \ $\sigma=\tau;$
\par {\indent\rm$(c)$} \ every idempotent pure congruence on $A$ is $E$-unitary$;$
\par {\indent\rm$(d)$} \ there exists an idempotent pure $E$-unitary congruence on $A;$
\par {\indent\rm$(e)$} \ $\tau$ is $E$-unitary.
\end{theorem}

\begin{proof} $(a)\implies (b)$. Let $\pi= 1_A$. Then $\sigma=\tau(\pi)=\tau(1_A)=\tau$.

$(b)\implies (a)$. Let $\sigma=\tau$. Then $\pi =\sigma_\kappa=\tau_\kappa=1_A$.

$(a)\implies (c)$. Firstly, $A/\tau$ is $E$-unitary. Indeed, if $(e\tau)(a\tau) = f\tau\in E_{A/\tau}$, where $a \in A$ and $e, f \in E_A$, then $(ea, f) \in\tau$.~Hence $ea \in E_A$.~Thus $a \in E_A$. Secondly, if $\rho\in\mathcal{C}(A)$ is idempotent pure, then $\rho\subseteq\tau$. Consequently, $\rho$ is $E$-unitary.

$(c)\implies (d)$. Obvious.

$(d)\implies (e)$. If $\rho$ is an idempotent pure $E$-unitary congruence on $A$, then we get $\pi\subseteq\rho\subseteq\tau\subseteq\sigma$, so $\tau$ is $E$-unitary (Theorem \ref{ME}).

$(e)\implies (a)$. Let $ea = f$, where $a\in A$ and $e,f\in E_A$. Then $(e\tau)(a\tau) = f\tau$ and so $a\in\text{ker}(\tau) = E_A$. Thus $S$ is $E$-unitary.
\end{proof}

\noindent
\bigbreak
\footnotesize{
Institute of Mathematics and Computer Science,

Wroclaw University of Technology

Wyb. Wyspianskiego 27

50-370 Wroclaw
Poland

e-mail: wieslaw.dudek@pwr.wroc.pl, \ \ romekgigon@tlen.pl}

\end{document}